\documentclass{amsart}

\usepackage{amsmath, latexsym, amssymb, amsthm, graphicx, color}

\usepackage{esint}

\begin{document}

\newcommand{\barint}{\overline{\hspace{.65em}}\!\!\!\!\!\!\int}
\newcommand{\ep}{\varepsilon}
\newcommand{\eps}{\varepsilon}
\newcommand{\ue}{u_\ep}
\newcommand{\ve}{v_\ep}
\newcommand{\vp}{{\varphi}}
\newcommand{\vez}{{v_\ep^z}}
\newcommand{\hep}{h_\ep}
\newcommand{\jep}{j_\ep}
\newcommand{\loc}{ {\mbox{\scriptsize{loc}}} }
\newcommand{\R}{{\mathbb R}}
\newcommand{\C}{{\mathbb C}}
\newcommand{\J}{{\mathbb J}}
\newcommand{\T}{{\mathbb T}}
\newcommand{\TL}{{\mathbb T}_{L}}
\newcommand{\Z}{{\mathbb Z}}
\newcommand{\N}{{\mathbb N}}
\newcommand{\Q}{{\mathbb Q}}
\newcommand{\Hdf}{{\mathcal{H}}}
\newcommand{\calD}{{\mathcal{D}}}
\newcommand{\calL}{{\mathcal{L}}}
\newcommand{\calW}{{\mathcal{W}}}
\newcommand{\calF}{{\mathcal{F}}}
\newcommand{\calG}{{\mathcal{G}}}
\newcommand{\calB}{{\mathcal{B}}}
\newcommand{\calH}{{\mathcal{H}}}
\newcommand{\calJ}{{\mathcal{J}}}
\newcommand{\calT}{{\mathcal{T}}}
\newcommand{\calM}{{\mathcal{M}}}
\newcommand{\dist}{\operatorname{dist}}
\newcommand{\spt}{\operatorname{spt}}
\def\rest{\hskip 1pt{\hbox to 10.8pt{\hfill
\vrule height 7pt width 0.4pt depth 0pt\hbox{\vrule height 0.4pt
width 7.6pt depth 0pt}\hfill}}}
\newcommand{\bd}{\partial}

\newcommand{\vol}{\,\mbox{vol}}

\newcommand{\bn}[1]{ ({\texttt{#1}})}

\newcommand{\beq}{\begin{equation}}
\newcommand{\eeq}{\end{equation}}

\def\logeps{{|\!\log\ep|}}

\theoremstyle{plain}
\newtheorem{theorem}{Theorem}
\newtheorem{proposition}{Proposition}
\newtheorem{lemma}{Lemma}
\newtheorem{corollary}{Corollary}

\theoremstyle{definition}
\newtheorem{definition}{Definition}

\theoremstyle{remark}
\newtheorem{remark}{Remark}
\newtheorem{example}{Example}
\newtheorem{warning}{Warning}

\title[ ]
{Dynamics of nearly parallel vortex filaments for the Gross-Pitaevskii equation}

\author{R.L. Jerrard}
\address{Department of Mathematics, University of Toronto, 
			Toronto, Ontario M5S 2E4, Canada.}
\email{rjerrard@math.utoronto.ca}

\author{D. Smets}
\address{Laboratoire Jacques-Louis Lions, Sorbonne Universit\'e, 
			4 place Jussieu BC 187, 75252 Paris Cedex 05, France.}
\email{smets@ann.jussieu.fr}

\begin{abstract}
In \cite{KMD}, Klein, Majda and Damodaran have formally derived a simplified
	asymptotic motion law for the evolution of nearly parallel vortex filaments
	in the context of the three dimensional Euler equation for incompressible
	fluids. In the present work, we rigorously derive the corresponding
	asymptotic motion law in the context of the Gross-Pitaevskii equation. 
\end{abstract}

\maketitle

%\tableofcontents

\section{Introduction}

The mathematical analysis of the evolution of vortex filaments within the
framework of the classical equations for fluids is a challenging problem that
dates back to the second half of the nineteenth century with the works of Kelvin
and Helmholtz. Some ``simplified'' flows have long been considered as potential
candidates for the description of the asymptotic regime of small vortex cores,
the most well-known being the binormal curvature flow of Da Rios over a century
ago, but the convergence proofs in all these cases are missing, and the validity
of the convergence is sometimes questioned too in the literature. 

\smallskip

In \cite{KMD}, Klein Majda and Damodaran have proposed the system 
\begin{equation}\label{KMD}
		\partial_t X_j = J \alpha_j \Gamma_j \partial_{zz} X_j + 
		J\sum_{k \neq j} 2 \Gamma_k \frac{X_j - X_k}{|X_j - X_k|^2}, 
		\qquad j=1,\cdots,n   
\end{equation}
as a simplified candidate model for the evolution of $n$ nearly parallel vortex
filaments in perfect incompressible fluids. This model extends a remark by
Zakharov \cite{Za1} for pairs of anti-parallel filaments, and is expected to be
valid only when  
\begin{enumerate}
	\item[$i)$] the wavelength of the filaments perturbations are large with
		respect to the filaments mutual distances, 
	\item[$ii)$] the latter are large with respect to the size of the filaments
		cores, and
	\item[$iii)$] the Reynolds number is sufficiently large.
\end{enumerate}
In the above formulation, the filaments are assumed to be nearly parallel to the 
z-axis, and after rescaling\footnote{Described further down, otherwise they
wouldn't be anything close to parallel!} each of them is described by a function
$z \mapsto (X_j(z,t),z)$, where $X_j(\cdot,t)$ takes values in $\R^2$, which
represents the horizontal displacement of the filament. The canonical two by two
symplectic matrix is denoted by $J$, the constants $\Gamma_j \in \R$ are the
circulations associated to each vortex filament, and the constants $\alpha_j \in
\R$ are derived from assumptions on the vortex core profiles prior to passing in
the limit.  

From the fluid mechanics point of view, the case $n=1$ in \eqref{KMD} is already
highly interesting and corresponds to a single weakly curved vortex filament.
In that case, system \eqref{KMD} reduces to the free Schr\"odinger equation in
one variable, and as a matter of fact this is also the linearized equation for
the binormal curvature flow around a straight filament.

From a mathematical point of view, system \eqref{KMD} has been studied for his
own (see e.g.  \cite{LiMa, KPV, BaMi1, BaFaMi}) when $n>1$, in particular its
well-posedness and the possibility of colliding filaments under \eqref{KMD}.
Nevertheless, as mentioned already, the justification of the model itself as a
limit from a classical fluid mechanics model (such as the Euler equation or the
Navier-Stokes equation in a vanishing viscosity limit) has so far only been
obtained formally through matched asymptotic, even for $n=1$.  

\smallskip

The goal the present work is to rigorously derive system \eqref{KMD}, for
arbitrary $n \geq 1$, as a limit from (yet another) PDE model whose relation to
fluid mechanics is not new. In that framework, all the limiting circulations
$\Gamma_j$ will end up being equal. Our object of study in this paper is indeed
the Gross-Pitaevskii equation 
\begin{equation}\label{gp}
	i \partial_t\ue - \Delta u + \frac 1{\ep^2}(|\ue|^2-1)\ue=0 
	\qquad\mbox{ in }(0,T)\times \Omega,
\end{equation}
with initial data $\ue(\cdot, 0) = \ue^0(\cdot).$ Here $0<\ep\ll 1$ is a real
parameter, $\Omega = \omega\times \TL$ where $\omega\subset \R^2$ is a bounded
open set with smooth boundary\footnote{Since a rescaling will eventually be made
in the description that sends the lateral boundary to infinity, the exact shape
of $\omega$ is of limited impact on the analysis, and the limit flow for the
filaments does {\it not} depend at all on $\omega.$ Still, some of our later
assumptions for establishing convergence do depend on $\omega$, see e.g.
\eqref{kappan.def}.} and $\TL = \R/{\rm L}\Z$ for some $L >0.$ Without loss of
generality, we shall assume that $0 \in \omega.$ We also consider Neumann
boundary conditions on 
$\partial \omega\times \TL$: 
\[
\nu \cdot \nabla u_\ep = 0 \mbox{ on }\partial\omega\times \TL.
\]

Our main result will describe solutions of \eqref{gp} associated to initial data
$\ue^0$ for vanishing families of $\eps$, and corresponding in a sense to be
described in detail below to $n$ nearly parallel vortex filaments clustered
around the vertical axis $\{0\}\times (0,L)$. 

\subsection{Statement of main result}
We consider the system 
\begin{equation}
	i \partial_t f_j - \partial_{zz} f_j  - 2\sum_{k\ne j} 
	\frac{f_j-f_k}{|f_j-f_k|^2} = 0, \qquad j=1,\ldots, n
\label{IVF}\end{equation}
for $f \equiv (f_1,\cdots, f_n)\: : \TL \times \R\to \C^n$. This is the Klein
Majda and Damodaran system \eqref{KMD} in the special case where all constants
are equal and normalized to unity.   

For $f\in H^1(\TL,\C^n)$, we define
\[
	G_0(f) :=\pi \int_0^L \left(\frac 12 \sum_{i=1}^n| f_i'|^2 -
	\sum_{i\ne j}\log|f_i-f_j| \right) dz,
\]
it is the Hamiltonian associated to the equation \eqref{IVF}. We also set 
\[
	\rho_f := \inf_{z\in (0,L), \ j\ne k}|f_j(z)-f_k(z)|.
\]
A sufficient condition for the Hamiltonian $G_0(f)$ to be finite is that $\rho_f
> 0.$ For $f^0 \in H^1(\TL,\C^n)$ such that $\rho_{f^0} > 0,$ system \eqref{IVF}
possesses a unique solution $f \in \mathcal{C}((-T, T), H^1(\TL,\C^n))$ for some
$T>0$, and which satisfies $\rho_{f(\cdot,t)} > 0$ for all $t \in (-T, T).$
Moreover, $f$ can be approximated by  (arbitrarily) smooth solutions of
\eqref{IVF}. If $\liminf_{t \to \pm T} \rho_{f(\cdot, t)} = 0,$ corresponding to
a collision between filaments, the possibility to extend the solution past $\pm
T$ is a delicate question, a situation which we won't consider in this work.    

Regarding the Ginzburg-Landau energy, we write points in $\Omega$ in the form
$(x,z)\in \omega\times \TL$, and define
\[
	e_\ep(u) 
	:=
	\frac 12\left( |\nabla_x\ue|^2 +  \frac 12 {|\partial_z\ue|^2} \right)+ 
	\frac {1} {4\ep^2}(|\ue|^2-1)^2 \ ,
\]
and 
\begin{equation}\label{Gep.def}
G_\ep(u):=\int_\Omega e_\ep(u)\,dx\,dz
- L \kappa(n, \ep, \omega)
\end{equation}
where $\kappa(n, \ep ,\omega) = n\pi \logeps + n(n-1)\pi|\log h_\ep| + O(1)$ is
defined more precisely in \eqref{kappan.def} below. The Cauchy problem for the
Gross-Pitaevskii equation is globally well posed for initial data with finite
Ginzburg-Landau energy (i.e. in $H^1(\Omega)$ here), and solutions can be
approximated by smooth ones too. 

The quantity which will define and locate the vorticity of a solution $u_\ep$ is
the (horizontal\footnote{The other two components of the 3D Jacobian also have
interpretations, see e.g. Proposition \ref{P.improve} below, but they do not
enter in the statement of our main theorem.}) Jacobian  
$$
 Ju_\ep := \nabla_x^\perp \cdot {{\rm Re} (u_\ep \nabla_x \overline u_\ep)}, 
$$
it is therefore a real function of $(x,z,t)$. 

In order to measure the discrepancy between vorticity and an indefinitely thin
filament, we will integrate in $z$ some norms on the slices $\omega \times
\{z\}.$ For 
$\mu \in W^{-1,1}(\omega)$ we let 
\[
\| \mu \|_{W^{-1,1}(\omega)} := \sup \left\{
\int \phi \,d\mu \ : \phi \in W^{1,\infty}_0(\omega), \  
\max\{ \|\phi\|_\infty, \|D\phi\|_\infty\} \le 1 \right\}.
\]
Among the various equivalent norms that induce the $W^{-1,1}(\omega)$ topology,
this choice has the property that there exists $r(\omega)>0$ such that if
$a_1,\ldots, a_n$ and $b_1,\ldots b_n$ are points in $B_r \subset \omega$, then 
\begin{equation}\label{BCL}
\| \sum_{i=1}^n \delta_{a_i} - \sum_{i=1}^n \delta_{b_i}\|_{W^{-1,1}(\omega)} = 
	\min_{\sigma \in S_n} \sum_{i=1}^n |a_i - b_{\sigma(i)}|
\end{equation}
where $S_n$ denotes the group of permutations on $n$ elements, see \cite{BCL}.
Indeed, this property holds whenever $r(\omega) \le \min\{ \frac 12
\mbox{dist}(0,\partial \omega),1 \}$, as then any $1$-Lipschitz function on
$B_r$ that equals zero at the origin can be extended to a function $\phi$ such
that $\phi = 0$ on $\partial \omega$ and $\max\{ \|\phi\|_\infty,
\|D\phi\|_\infty\} \le 1$. 

\smallskip
Finally, we introduce the scale 
$$
h_\ep := \frac{1}{\sqrt{\logeps}}.
$$
It will correspond to the amount of deformation of the filaments with respect to
perfectly straight ones, and is also the typical separation distance between
distinct filaments.  At the same time, the scale $\eps$ corresponds to the
typical core size of the filaments, and therefore since $h_\ep \gg \ep$ as $\eps
\to 0,$ the displacements and mutual distances of filaments are much larger in
this asymptotic regime than their core size.   

\medskip

Our main result is

\begin{theorem}\label{T.main}
Let $f  =  (f_1,\ldots, f_n) \in \mathcal{C}((-T,T),H^1(\TL,\C^n))$ be
solution of the vortex filament system \eqref{IVF} with initial data $f^0$
and such that $\rho_{f(t)} \ge \rho_0>0$ for all $t\in (-T,T)$, 
and $\partial_t f \in L^\infty((-T,T) \times \TL)$. 

For $\ep\in (0,1]$, let  $u_\ep$  solve the  Gross-Pitaevskii equation
\eqref{gp} for initial data such that 
\begin{equation}
	\int_0^L \Big\| 
	J_xu_\ep^0(\cdot,z) - \pi \sum_{j=1}^n \delta_{h_\ep f_j^0(z)}
	\Big\|_{W^{-1,1}(\omega)} dz \  =  \ o(h_\ep)
	\label{mainthm.h1}
\end{equation}
and 
\begin{equation}
	G_\ep(u_\ep^0)\rightarrow G_0(f^0) 
	\label{mainthm.h2}
\end{equation}
as $\ep \to 0$.
Then for every $t\in (-T,T)$,
\begin{equation}
	\int_0^L \Big\| 
	J_xu_\ep(\cdot,z, h_\ep^2 t) - \pi \sum_{j=1}^n \delta_{h_\ep f_j(z,t)}
	\Big\|_{W^{-1,1}(\omega)} dz \  =  \ o(h_\ep),
	\label{mainthm.c1}
\end{equation}
as $\ep \to 0$.
\end{theorem}

\noindent{\bf Comments.} 
The positivity of $\rho_0$ in Theorem \ref{T.main} is essential, 
it implies that no 
collision between filaments occured over time, and the corresponding 
conclusion would very likely be incorrect without assuming it. Indeed, filaments
collisions in real fluids experiments was observed to lead to highly complex 
reconnection dynamics. The uniform bound assumption on 
$\partial_t f$ may be only technical, and it is verified if e.g. $f^0 \in H^s(\TL)$ for
some $s > 5/2$, in view of the positivity of $\rho_0$ and classical Cauchy theory for
the Schr\"odinger operator in one space dimension. Assumption \eqref{mainthm.h1} 
is responsible for the concentration 
of the initial vorticity of $u_\ep$ around the filaments parametrized by (rescalings)
of $f^0$.  Assumption \eqref{mainthm.h2} can be
understood as requiring that the former concentration holds in the most energy
efficient way (at least asymptotically as $\ep \to 0$); this follows from results in
\cite{CJ}, building on earlier work of \cite{DPKo}. Below we will recall these results
in detail and refine some of them.
The conclusion
\eqref{mainthm.c1} implies that the concentration of vorticity is preserved in time,
and its location follows (after appropriate rescalings) the model of Klein Majda
and Damodoran.

\smallskip
In the context of the 3D Gross-Pitaevskii equation, there are very few available
mathematical results which rigorously derive a motion law for vortex filaments.
Besides Theorem \ref{T.main}, the only one we are aware of which does not
require a symmetry assumption reducing the actual problem to 2D is
\cite{filaments},
where the case of a single vortex ring was treated (the limiting filament is
symmetric but the field $u_\ep$ is not assumed to be so).  The situation is
slightly better understood in the axisymmetric setting, in particular the case of a
finite number of vortex rings was analized in \cite{JeSm4}, where the so-called
leapfrogging phenomenon was established. In 2D the situation is of course 
brighter, and since vortex filaments are for the most part tensored versions of
2D vortex points, it is not surprising that the analysis of the latter is at 
the basis of all the 3D  works we were referring to so far. 

Vortex points and approximations of in 2D evolve according to the so-called point 
vortex system. That was established in \cite{JeCo} in the context of the 
Gross-Pitaevskii equation, but parallel results were also obtained (and actually
earlier) in the framework of the incompressible 2D Euler equation \cite{MaPu1,
MaPu2}.

The analogy between Euler and Gross-Pitaevskii equations is expected to be 
valid not only in 2D, and as stated at the beginning of this introduction a 
common open challenge in both frameworks is to rigorously derive the binormal 
curvature flow equation for general vortex filament shapes.
In this context, we emphasize the $n=1$ case of Theorem \ref{T.main} 
establishes a linearized version of this so-called self-induction approximation for \eqref{gp};
the general case  of the theorem
describes evolution governed by a combination of the linearized self-induction
of filaments and interaction with other filaments.

Contrary to the Euler equation, the Gross-Pitaevskii equation has a fixed 
``core length" $\ep$ in its very definition: this simplifies some of the analysis 
and may explain
why in particular the equivalent of the nonlinear 3D stability for one vortex ring
or the leapfrogging phenomenon have not yet been proved in that context. On the
other hand, there is no equivalent of the Biot-Savart law in the context of the
Gross-Pitaveskii equation, the field is complex and the analysis often
involves tricky controls of the phases. Partial results in the context of Euler
in 3D include \cite{GaSm1, GaSm2} for the 3D spectral stability of a columnar 
vortex, and \cite{BuMa} for the evolution of a finite number of axisymmetric 
vortex rings in a regime where they do not interact.

Theorem \ref{T.main} does not cover the case of anti-parallel vortex filaments,
a situation which in \eqref{KMD} would correspond to constants $\Gamma_j \in \pm
1$ that do not all share the same sign. This is something that wish to consider
in the future.

\bigskip

In the remaining subsections of this introduction, after fixing a number of
notations which we use throughout, we describe in details the strategy followed
to prove Theorem \ref{T.main} and we state the key intermediate lemmas and
propositions. The proofs of the latter are presented latter in Section
\ref{sec:dynamics}, for the key arguments related to the dynamics, in Section
\ref{App}, for the results which do not depend on a time variable and which are
for the most part extensions or variations of results in \cite{CJ}, and in
Section \ref{sec:compact}, for those related to a priori compactness in time. 

\subsection{Further notation}\label{sec:furnot}

In addition to the scale $h_\ep := \logeps^{-1/2}$, we will always write 
$\omega_\ep := h_\ep^{-1}\omega$ and $\Omega_\ep := \omega_\ep \times \TL$ 
to denote the rescaled versions of $\omega$ and $\Omega$
respectively. Given  
$u_\ep \in H^1(\Omega,\C)$ we will always let $v_\ep$ denote the function
in $H^1(\Omega_\ep,\C)$ defined by
$$
	v_\ep(x,z) = u_\ep(h_\ep x, z), \qquad (x,z)\in \Omega_\ep.
$$ 
We will write 
\[
	jv_\ep := iv_\ep \cdot \nabla_x v_\ep ,
\]
where here and throughout, a dot product of complex numbers denotes the {\em
real} inner product: 
\[
	\mbox{ for }v,w\in \C, \qquad v\cdot w = \mbox{Re}( v\bar w) .
\]
Observe once more that $j v_\ep$ contains only the {\em horizontal} components
of the momentum vector $iv_\ep \cdot D v_\ep = (iv_\ep \cdot \nabla_x
v_\ep,iv_\ep \cdot \partial_z v_\ep)$.

In many places, we implicitly identify $\C^n$ with $(\R^2)^n$ when no complex
products are involded. We fix $\chi \in C^\infty(\R)$ to be a nonnegative
nonincreasing function such that 
\[
	\chi(s) = 1\mbox{ if }s<1, \qquad \chi(s)=0\mbox{ if }s\ge 2, 
\]
and for arbitrary $r>0$ we set $\chi_r(s):= \chi(s/r).$ For $f\in H^1((0,L),
(\R^2)^n)$ such that $\rho_f>0$, and for $0<r<\rho_f/4,$ we also set
\[
	\begin{aligned}
	\chi_{r}^f (x,z) &:= \sum_{i=1}^n \chi_r( |x-f_i(z)| )  \, |x-f_i(z)|^2. \\
	\chi_{r,\ep}^f (x,z) &:= \frac 1{h_\ep^2} \chi^{h_\ep f}_{h_\ep r}(x,z)
	= \sum_{i=1}^n \chi_r(\frac {|x - h_\ep f_i(z)|}{h_\ep} )
		\, \left|\frac {x -h_\ep f_i(z)}{h_\ep}\right|^2.
	\end{aligned}
\]

Repeated indices $a,b,c,\ldots$ are implicitly summed from $1$ to $2$; these
correspond to the horizontal $x$ variables. We will also write $\ep_{ab}$ to
denote the usual antisymmetric symbol, with components
\[
\ep_{12}=-\ep_{21} = 1, \qquad \ep_{11} = \ep_{22}=0.
\]
For $v = (v_1,v_2)\in \R^2$, we will write $v^\perp := (-v_2,v_1)$. Thus
$(v^\perp)_b = \ep_{ab}v_a$. We will similarly write $\nabla_x^\perp :=
(-\partial_y,\partial_x)$. In the same spirit,  
\[
	v^\perp := (v_1^\perp,\ldots,v_n^\perp)
	\quad\mbox { for }v = (v_1,\ldots, v_n)\in (\R^2)^n,
\]
with a similar convention for $\nabla^\perp W$, for $W:(\R^2)^n\to \R$.

If $\mu_z$ is a family of signed measures on an open set   $U\subset \R^2$,
depending (measurably) on a parameter $z\in (0,L)$, then $\mu_z\otimes dz$
denotes the measure on $U\times (0,L)$ defined by 
\[
	\int_{U\times (0,L)} f d\mu_z\otimes dz 
	= \int_0^L (\int_U f(x,z) d\mu_z(x)) dz. 
\]

For a smooth bounded  $A\subset \R^2$ (typically $\omega$ or $\omega_\ep$) and
$a\in A^n$ we will write
\[
	j^*_A(x;a) :=  -\nabla^\perp_x \psi_{A}^*,
\]
where $\psi_{A}^* = \psi_A^*(x;a) $ solves
\[
	\begin{cases}
	-\Delta_x\psi_{A}(\cdot; a)^* &= 2\pi \sum_{i=1}^n \delta_{a}
	\quad \mbox{ in } A\\
	\hspace{2em}\psi_A^*&=0\hspace{5.3em}\mbox{ on }\partial A\ .
	\end{cases}
\]
Equivalently, $j^*_A(x;a): A\to \R^2$ is the unique solution of
\begin{align*}
	\nabla_x \cdot j^*_A  =0, \qquad
	\nabla_x^\perp \cdot j^*_A = 2\pi \sum_{i=1}^n \delta_{a_i} ,\qquad
	j^*_A(\cdot, a)\cdot \nu = 0\mbox{ on }\partial A
\end{align*}
where $\nu$ denotes the  outer unit normal to $A$. It is straightforward to
check that 
\[
	j^*_{\omega_\ep}(x; a) = h_\ep j^*_\omega(h_\ep x; h_\ep a) 
\]
and that 
\[
	\lim_{\ep\to 0} j^*_{\omega_\ep}(x; a) 
	= \sum_{i=1}^n \frac{(x-a_i)^\perp}{|x-a_i|^2} =: j^*_{\R^2}(x;a).
\]
Given $g: (0,L)\to A^n$, we will write $j^*_A(g)$ to denote the function
$A\times (0,L)\to \R^2$ defined by
\[
	j^*_A(g)(x,z) = j^*_A(x; g(z)).
\]

We define a couple of other auxiliary functions related to $\psi_A$. First, note
that 
\[
	\psi_A(x;a) = - \sum_{i=1}^n\left( \log|x-a_i| + H_A(x, a_i)\right)
\]
where for $a_i\in \Omega$, we define $H_A(\cdot ,a_i)$ to be the solution of
\[
	-\Delta_x H(x, a_i) = 0\mbox{ for }x\in A, \qquad H_A(x,a_i) 
	= -\log|x-a_i|\mbox{ for }x\in \partial A.
\]
We define
\[
	W_A(a) = -\pi \Big(\sum_{i\ne j} \log|a_i-a_j| 
	+ \sum_{i,j} H_A(a_i,a_j)\Big).
\]
The constant $\kappa(n,\ep,\omega)$ appearing in \eqref{Gep.def} is defined by
\begin{equation}\label{kappan.def}
	\kappa_n(\omega) =  n(\pi \logeps +\gamma) 
	+ n(n-1)\pi |\log h_\ep| - \pi n^2 H_\omega(0,0) 
\end{equation}
where $\gamma$ is a universal constant\footnote{We will not need the exact
definition of $\kappa_n$ or $\gamma$ in this paper, but these constants will
appear in various formulas.} introduced in the pioneering work of B\'ethuel,
Brezis and H\'elein \cite{BBH}, see Lemma IX.1.

\subsection{Variational aspects of nearly parallel vortex
filaments}

In this section we first collect some information about the behaviour of nearly
parallel vortex filaments under energy and localisation constraints, but without
introducing any time dependence.  Most of these results are contained in
Contreras and J. \cite{CJ}, or can be obtained by adapting and combining results
in \cite{CJ}. The necessary details are given in Section \ref{App}. 

Our first result follows directly from arguments in \cite{CJ}, although it does
not appear there in exactly this form. 

\begin{proposition}\label{P:CJ}
Assume that $(u_\ep)\subset H^1(\Omega,\C)$ is a sequence satisfying 
\begin{align}
	\int_0^L \| J_x u_\ep(\cdot ,z) 
	- n\pi\delta_0\|_{W^{-1,1}(\omega)}dz &\le c_1 h_\ep,
	\label{CJ.h1}\\
	G_\ep(u_\ep) &\le c_2.
	\label{CJ.h2}
\end{align}
Then 
\begin{equation}\label{dzu2}
	\int_\Omega  |\partial_z u_\ep|^2 dx\;dz \leq C(c_1,c_2)
\end{equation}
and there exists some $f = (f_1,\ldots, f_n)\in H^1(\TL, \C^n)$ such that after
	passing to a subsequence if necessary: 
\begin{equation}\label{comp}
	\int_0^L \| J_x u_\ep(\cdot ,z) 
	- \pi\sum_{j=1}^n \delta_{h_\ep f_j(z)}\|_{W^{-1,1}(\omega)} dz = o(h_\ep)
	\qquad
	\mbox{ as }\ep\to 0.
\end{equation}
Finally, $f$ satisfies
\begin{equation}\label{fbounds}
	G_0(f)\le  \liminf_{\ep\to 0}G_\ep(u_\ep) , 
	\qquad\qquad \| f\|_{H^1} \le C(c_1,c_2), 
\end{equation}
where the lim inf refers to the subsequence for which
	\eqref{comp} holds.
\end{proposition}

%The second one provides additional information on limits of some
%vector measure, as well as localized $\Gamma$-lim inf type inequality. 
The arguments needed to extract Proposition \ref{P:CJ} from facts established in
\cite{CJ} are presented in Section \ref{subsec:pp1}. Next we describe
weak limits of products of derivatives of $v_\ep$.

\begin{proposition}\label{P.improve}
Assume that $(u_\ep)\subset H^1(\Omega,\C)$ satisfies \eqref{CJ.h2} and
	\eqref{comp} (and hence \eqref{CJ.h1}), and let $v_\ep(x,z) = u_\ep(h_\ep x,
	z)$. Then the following hold, in the weak sense of measures on $\Omega$ 
\begin{align}
	\frac 1 \logeps\partial_{x_k} v_\ep  \cdot\partial_{x_l} v_\ep 
	&\rightharpoonup \pi  \delta_{kl} 
	\sum_{i=1}^n \delta_{f_i(z)} \otimes dz,
	\label{improve1}\\
	\frac 1 \logeps \nabla_x v_\ep \cdot \partial_z v_\ep  
	&\rightharpoonup  
	-\pi  \sum_{i=1}^n \partial_z f_i (z)\delta_{f_i(z)} \otimes dz
	\label{improve2},
\end{align}
for all $k,l$ in $\{1, 2\}.$ Moreover, for any nonnegative $\phi \in
	C_c(\R^2\times \TL)$,
\begin{equation}\label{local.dz2.bound}
	\liminf_{\ep\to 0} 
	\int_{\Omega_\ep} \phi \frac{|\partial_z v_\ep|^2}\logeps \, dx\,dz \ge
	\pi \sum_{i=1}^n\int_0^L |f_i'(z)|^2 \phi(f_i(z), z) \, dz.
\end{equation}
\end{proposition}

The proof of Proposition \ref{P.improve} is given in Section \ref{subsec:pp2}.
Briefly, \eqref{improve1} and \eqref{local.dz2.bound} are deduced by combining
results from \cite{CJ} with facts established in \cite{wave, filaments, SandSerf-product},
and \eqref{improve2} is obtained via a short argument whose starting point
is \eqref{improve1} and \eqref{local.dz2.bound}.

Finally we will need a refinement of a $\Gamma$-limit lower bound from \cite{CJ}.
The proof is given in Section \ref{subsec:pp3}.

\begin{proposition}\label{prop:3}
Let $r>0$ and $f\in H^1((0,L), \C^n)$ be given such that  $r<\rho_f/4.$ Then
	given $\delta>0$, there exist $c_3, \ep_3>0$, depending only on
	$\|f\|_{H^1}$ and $r$, such that for all $\Sigma\in (0,1]$ and any $\ep\leq
	\ep_3$, if $u_\ep\in H^1(\Omega,\C)$ and 
\begin{equation}\label{lessthanc2}
	\int_0^L \| Ju_\ep(\cdot, z) 
	- \pi \sum_{i=1}^n \delta_{h_\ep f_i(z)}\|_{W^{-1,1}(\omega)} dz
	\le c_3 h_\ep ,
\end{equation}
\begin{equation}
	\label{eq:nakaayi}
	G_\ep(u_\ep) - G_0(f) \le \Sigma , 
\end{equation}
then
\begin{equation}\label{eq:galibier}
	\int_0^L \int_{\omega \setminus \cup_{i=1}^n B(h_\ep f_i(x), h_\ep r)}
	e_\ep(|u_\ep|) + \frac14 \left|\frac  {ju_\ep}{|u_\ep|}  
	- j^*_\omega(h_\ep f)\right|^2 \le K_3\Sigma  +\delta,
\end{equation}
where $K_3$ depends only on $r , n$, and $\|f\|_{H^1}$. Moreover, if
\begin{equation}\label{eq:izoard}
	T^f_{r,\ep}(u_\ep) := 
	\int_\Omega J_x u_\ep(x,z) \chi^f_{r,\ep} \,dx\,dz 
	\le \frac{c_3^2}{4n\pi L}
\end{equation}
then
\begin{equation}\label{eq:lautaret}
	\frac 1 {h_\ep}\int_0^L \|J_x u_\ep(\cdot,z) 
	- \pi\sum_{i=1}^n  \delta_{h_\ep f_i(z)}\|_{W^{-1,1}(\omega)} dz 
	\le \left(n\pi L T^f_{r,\ep}(u_\ep)\right)^\frac12  + o(1) \le \frac 12 c_3.
\end{equation}
\end{proposition}

\subsection{Compactness in time}\label{sec:compact0}

In this section we now assume that $u_\ep$ is a solution of the Gross-Pitaevskii
equation and we shall obtain sufficient compactness in time to pass to the limit
as $\ep \to 0$ on intervals of time of positive length. 

\begin{proposition}\label{prop:4}
Let $r>0$ and $g\in W^{1,\infty}(\TL, \C^n)$ be given such that $r\leq
	\rho_g/4.$ There exist $\ep_4,\,c_4 >0$, depending only on $\|g\|_{H^1}$ and
	$r$, and there exist $C_4$,  depending only on $\|g\|_{Lip}$ and $r$, with
	the following properties. If $u_\ep$ solves the Gross-Pitaevskii equation
	\eqref{gp} for some $0<\ep \leq \ep_4$ for initial data $u_\ep^0$ satisfying 
\begin{equation}\label{eq:barguil1}
	G_\ep(u_\ep^0) \leq G_0(g)+1,
\end{equation}
\begin{equation}\label{eq:barguil2}
	\int_0^L \| Ju_\ep^0(\cdot, z) 
	- \pi \sum_{i=1}^n \delta_{h_\ep g_i(z)}\|_{W^{-1,1}(\omega)} dz
	\le c_4 h_\ep, 
\end{equation}
and
\begin{equation}\label{eq:barguil3}
	T^g_{r,\ep}(u_\ep^0) \leq \frac{c_4^2}{4n\pi  L},
\end{equation}
then for every $0\leq t \leq t_4 := 3c_4^2 /(4 C_4n\pi L)$,
	\begin{equation}\label{eq:bardet1}
	T^g_{r,\ep}(u_\ep(\cdot,\cdot,h_\ep^2t)) \leq T^g_{r,\ep}(u_\ep^0) + C_4t,
\end{equation} 
\begin{equation}\label{eq:bardet2}
	\frac 1{h_\ep}\int_0^L \|J_x u_\ep(\cdot,z,h_\ep^2t) 
	- \pi\sum_{i=1}^n  \delta_{h_\ep g_i(z)}\|_{W^{-1,1}(\omega)} dz  
	\le  \left(n\pi L (T^g_{r,\ep}(u_\ep^0)+C_4 t)\right)^\frac12 +o(1),
\end{equation}
and in particular
\begin{equation}\label{eq:bardet3}
	\int_0^L \|J_x u_\ep(\cdot,z,h_\ep^2t) 
	- \pi\sum_{i=1}^n  \delta_{h_\ep g_i(z)}\|_{W^{-1,1}(\omega)} dz  
	\le  c_4 h_\ep .
\end{equation}
\end{proposition}

The proof is given in Section \ref{sec:compact}, as is the proof of the following.

\begin{corollary}\label{cor:1}
Under the assumptions of Theorem \ref{T.main}, there exists $t_0>0$, depending
	only on $\rho_{f^0}$ and $\|f^0\|_{H^1}$, $f^*$ in
	$\mathcal{C}([0,t_0],L^1(\TL,\C^n)) \cap L^\infty([0,t_0],H^1(\TL,\C^n))$,
	and a common sequence $\ep \to 0$, such that for every $0\leq t \leq t_0$ 
$$
	\int_0^L \| J_x u_\ep(\cdot ,z,h_\ep^2t) - 
	\pi\sum_{j=1}^n \delta_{h_\ep f^*_j(z,t)}\|_{W^{-1,1}(\omega)} dz = o(h_\ep)
	\qquad
	\mbox{ as }\ep\to 0
$$
and in addition \eqref{eq:bardet3} holds for all $t\in [0,t_0]$, for every $\ep$
	in the sequence. 

Moreover, we have $f^*(0) = f(0)$ and
\begin{equation}\label{whoisbardet}
	\sup_{s,t\in [0,t_0]}\max_{i,z} |f^*_i(z,t) - f_i(z,s)| 
	\le \frac{\rho_0}8,\qquad\mbox{ and hence }
	\inf_{t \in [0,t_0]}\rho_{f^*(t)} \geq \frac 34 \rho_0.
\end{equation}
\end{corollary}

Our main goal in the sequel is to show that $f$ and $f^*$ coincide on $[0,t_0]$,
from which Theorem \ref{T.main} will follow by a straightforward continuation
argument.

\begin{proposition}\label{P:jstar}
In addition to the statements in Corollary \ref{cor:1}, we have
$$
	\frac{j(v_\ep)}{|v_\ep|} \rightharpoonup j_{\R^2}^*(f^*) \quad\text{weakly in } L^2(O) 
$$ 
for every open $O \subset\subset \{ (t,x,z)\in [0,t_0]\times \R^2 \times \TL :
	x\ne f^*_k(z,t), \ k=1,\ldots, n\}$.
\end{proposition}

\subsection{Proof of the main theorem}

For points $a = (a_1,\ldots, a_n)\in (\R^2)^n$ such that $a_i\ne a_j$ for $i\ne
j$, we will write
\begin{equation}\label{calW.def}
\calW(a) = -\sum_{i\ne j} \log|a_i-a_j|.
\end{equation}
With this notation, 
\[
	G_0(g) = \pi \int_0^L \frac 12 |g'(z)|^2 
	+ \calW(g(z)) \, dz \quad\mbox{ for }g:\TL\to (\R^2)^n.
\]
For $0\le t \le t_0$ (where $t_0$ appears in  Corollary  \ref{cor:1}), we define
\begin{align*}
	I_1(t) &:= \pi \int_0^L |f(z,t) - f^*(z, t)|^2 \, dz\\
	I_2(t) &:=  \pi \int_0^L  \left(-\partial_{zz}f(z,t) 
	+ \nabla \calW(f(z,t) \right) \cdot (f(z, t) - f^*(z, t)) dz\\
	I_3(t) &:= G_0(f(\cdot, t)) - G_0(f^*(\cdot, t)).
\end{align*}
Note that, as a consequence of conservation of energy for both \eqref{gp} and
\eqref{IVF},
\[
	G_0(f(\cdot, t)) = G_0(f^0) = \lim_{\ep\to 0}G_\ep(u_\ep^0)
	=
	\lim_{\ep\to 0}G_\ep(u_\ep(t)) \ge G_0(f^*(t)).
\]
The last inequality follows from \eqref{fbounds}, as discussed following the
statement of Proposition \ref{P:CJ}. Thus $I_3(t)\ge 0$ for all $t\in [0,t_0]$.
In addition, $I_3(0)=0$, due to
\eqref{mainthm.h2}.

We aim to apply Proposition \ref{prop:3} to control quantities such as
$\frac{ju_\ep}{|u_\ep|}(t) - j_\omega^*(h_\ep f^*(t))$ for a range of $t$. To
this end, we will need 
\[
	\Sigma_\ep(t) := G_\ep(u_\ep(t)) - G_0(f^*(t)) \le 1.
\]
Arguing as above, we see that $ \lim_{\ep\to 0}\Sigma_\ep(t) = I_3(t)$.
Thus $\Sigma_\ep(t)\le 1$ if $\ep$ is sufficiently small and $I_3(t)\le \frac 12$.
We therefore define
\[
	t^* := \sup \{ t\in [0,t_0] : 0\le  I_3(s) 
	\le \frac 12 \mbox{ for all }s\in [0,t]\}.
\]
The positivity of $t^*$ is a consequence of  the weak $H^1$ lowersemicontinuity
of $f\mapsto G_0(f)$ and the continuity properties of $f^*$ as stated in
Corollary \ref{cor:1}. (The other hypothesis of Proposition \ref{prop:3} follows
directly from Corollary \ref{cor:1}.) 

Theorem \ref{T.main} will be an easy consequence of the following three lemmas.

\begin{lemma}\label{lem:pmt1}
There exists a constant $C_2$ such that for every $t\in [0,t^*]$, 
\[
	I_3(t) \le   I_2(t) + C_2I_1(t).
\]
\end{lemma}

\begin{proof}
First, it follows from   \eqref{whoisbardet} that for every $z\in [0,L]$ and
	$t\in [0,t^*]$,
\[
	\calW(f(z,t)) - \calW(f^*(z,t)) 
	\le \nabla \calW(f(z,t))\cdot(f(z,t) - f^*(z,t)) + C|f(z,t) - f^*(z,t)|^2,
\]
for $C$ depending only on $\rho_{f(0)}$. The conclusion of the lemma follows by
	integrating this inequality with respect to $z$ and combining the result
	with the estimate 
\begin{align*}
	\frac \pi 2\int_0^L |\partial_z f|^2 - |\partial_z f^*|^2 \, dz
	&=
	\frac \pi 2\int_0^L 2\partial_z f \cdot\partial_z(f-f^*) 
	- |\partial_z (f-f^*)|^2 \, dz
	\\
	&\le
	- \pi \int_0^L \partial_{zz}f \cdot (f^*-f).
\end{align*}
\end{proof}

The proofs of the next two lemmas are presented in Section \ref{sec:dynamics} below.
\begin{lemma}\label{lem:pmt2}
For every $\tau\in [0,t^*]$, 
\[
	I_1(\tau) \le I_1(0) +  C\int_0^\tau \left(I_1(t) +  I_3(t)\right) dt.
\]
\end{lemma}

\begin{lemma}\label{lem:pmt3}
For every $\tau\in [0,t^*]$, 
\[
	I_2(\tau) \le I_2(0) + C \int_0^\tau  \left(I_1(t) +  I_3(t)\right) dt.
\]
\end{lemma}

With these, we can complete the 

\begin{proof}[{\rm \bf Proof of Theorem \ref{T.main}}]
Let $I_4(t) = I_2(t)+C_2I_1(t)$. It follows from Lemma  \ref{lem:pmt1} that
	$I_4(t)\ge 0$ for all $t\in [0,t^*]$, and Lemmas \ref{lem:pmt1} --
\ref{lem:pmt3} imply that
\[
	I_4(\tau) \le  C\int_0^\tau I_4(t)\, dt 
	\qquad\mbox{ for all }\tau \in [0,t^*].
\]
It follows by Gr\"onwall's inequality that $I_4(\tau) = 0$ for all $\tau \in
	[0,t^*]$.  Thus by Lemma \ref{lem:pmt1}, we have that $I_3=0$ on $[0,t^*]$.
	Then from Lemma \ref{lem:pmt2} and another invocation of Gr\"onwall, we
	conclude that $I_1=0$ on $[0,t^*]$, in other words, that $f=f^*$ on $[0,
	t^*]$. A straightforward continuation argument now  shows that this equality
	holds on $(0,T)$, and then by reversibility on $(-T,T)$, thus completing the
	proof.
\end{proof}

\section{Dynamics}\label{sec:dynamics}

The object of this section is to present the proofs of
Lemma \ref{lem:pmt2} and Lemma \ref{lem:pmt3}, from which (together 
with Lemma \ref{lem:pmt1}) our main Theorem was
derived in the Introduction. We will find it useful to rescale the
Gross-Pitaevskii equation \eqref{gp}, setting
\begin{equation}
\ve(x,z,t) := \ue(h_\ep x, z, h_\ep^2 t),
\label{rescale.witht}\end{equation}
where
\[
h_\ep := \logeps^{-1/2}.
\]
Thus
\begin{equation}\label{gpv}
i \partial_t\ve - \Delta_x \ve - \frac{\partial_{zz}\ve}{\logeps} + \frac
{1}{\logeps\ep^2}(|\ve|^2-1)\ve = 0.
\end{equation}

We will write
\begin{align*}
j_x\ve &:= i\ve\cdot \nabla_x\ve,\\
j_z\ve &:= i\ve\cdot \partial_z\ve.
\end{align*}
For the rescaled equation \eqref{gpv}, the equation for conservation of mass
takes the form
\begin{align}
\label{massv}
	\frac 12 \partial_t |\ve|^2 
	&=  \nabla_x \cdot  j_x \ve + \hep^2\, \partial_z j_z\ve. 
\end{align}
We will rely mainly on the equation for vorticity, and in fact only for the $z$
component of the vorticity vector, which is precisely $J_x v_\ep$. By rescaling
standard identities we have
\[
\partial_t J_x\ve = 
 \ep_{ab} \partial_{a c}( \partial_b\ve \cdot \partial_c \ve)
+ 
\ep_{ab} \partial_{a z} (\frac {\partial_b\ve \cdot \partial_z \ve}\logeps ).
\]
Thus,
\begin{align}\label{dtJzv.int}
	\frac d{dt}\int \vp J_x\ve dx\,dz
	&=
	\int \partial_t \vp \, J_x v_\ep dx\, dz
	+
	\int \ep_{ab}\partial_{ac} \vp  \  \partial_b\ve \cdot \partial_c\ve \,dx\,dz \\
	&\hspace{3em}+
	\int \ep_{ab}\partial_{az}\vp  \frac{\partial_b \ve\cdot 
	\partial_z \ve}\logeps \,dx\,dz,
	\nonumber
\end{align}
for smooth $\vp:\Omega_\ep\times(0,T)\to \R$ for some $T>0$, with compact
support in $\Omega_\ep = \omega_\ep \times \TL$. (That is, test functions are
only required to have compact support with respect to the horizontal $x$
variables, not the periodic $z$ variable.) 

\begin{lemma}\label{lem:ptl}
Assume that  $\vp \in C^2_c(\Omega_\ep\times [0,t^*] )$ is a function such that 
for some $k\in \{1,\ldots, n\}$,
\[
	\mbox{supp}(\vp) \subset
	\{ (x,z,t) : |x-f_k(z,t)|\le \frac {\rho_0}2\},
\]
and
\begin{equation}\label{dac}
	\partial_{ac}\vp(x,z,t) 
	= c(z,t)\delta_{ac}\qquad\mbox{ in }\{ (z,t) : |x-f_k(z,t)|\le \frac {\rho_0}4\}
\end{equation}
for some continuous $c(z,t)$. Assume also that 
\begin{equation}\label{helpful}
	\sup_t \sup_z \| \vp(\cdot, z,t)\|_{C^1(\omega_\ep)} , \ \ 
	\sup_t \sup_z \| \partial_t\vp(\cdot, z,t)\|_{C^1(\omega_\ep)} , \ \ 
	\sup_t \| \partial_z\nabla_x\vp\|_{L^\infty(\Omega_\ep)}\le C.
\end{equation}
Then for any $\tau\in [0,t^*]$, 
\[
	\begin{aligned}
	&\int_0^L\vp(f^*_k(z,t), z, t)\,dz \Big|_{t=0}^{t=\tau}\\
	&\le
	C \int_0^\tau I_3(t)\,dt + \int_0^\tau \int_0^L 
		\partial_t\vp( f^*_k(z,t), z,t)\, dz\, dt \\
	&\qquad
	-\int_0^\tau \int_0^L \nabla^\perp\partial_z\vp(f_k^*(z),z,t)
		\cdot  \partial_zf^*_k(z,t) \, dz\,dt\\
	&\qquad
	+\int_0^\tau\int_0^L\nabla\vp(f_k^*(z,t),z,t)
		\cdot \nabla_k^\perp \calW(f^*(z,t)) \, dz\,dt\ .
	\end{aligned}
\]
\end{lemma}
\begin{proof}
We apply \eqref{dtJzv.int} to $\vp$, integrate both sides from $0$ to $\tau$,
	and send $\ep \to 0$.  We consider the various terms that arise. 

{\bf 1}.
Assumption \eqref{mainthm.h1} and properties of the support of $\vp$
imply that
\begin{equation}\label{lim1}
	\lim_{\ep\to 0}  \int_{\Omega_\ep} \vp(x,z,t) \, J_x v_\ep(x,z,t) dx \, dz =
	\pi\int_0^L\vp(f^*_k(z,t), z, t)\,dz
\end{equation}
for every $t\in [0,t^*]$, and in particular for $t=0,\tau$.

{\bf 2}.
Similarly, \eqref{lim1} holds with $\vp$ replaced by $\partial_t\vp$. In
addition, it follows from \eqref{helpful} and \eqref{eq:bardet3} that
$|\int_{\Omega_\ep} \vp_t(x,z,t) \, J_x v_\ep(x,z,t) dx \, dz|$ is bounded
uniformly in $t$.  Thus 
\[
	\lim_{\ep \to 0} \int_0^\tau\int_{\Omega_\ep}\partial_t \vp 
	\, J_x v_\ep \,dx\,dz\,dt
	= 
	\pi \int_0^\tau \int_0^L \vp(f_k^*(z,t), z,t) \, dz\,dt.
\]

{\bf 3}. 
The last term on the right-hand side of \eqref{dtJzv.int} is similar. First note
	that there exists some $C$ such that 
\[
	\int_{\Omega_\ep}
	\ep_{ab}\partial_{az}\vp  \frac{\partial_b \ve\cdot \partial_z \ve}\logeps \,dx\,dz\
	\le C
\]
for every $t$. This is a consequence of  \eqref{dzu2} (which is available for
	all $t\in [0,t^*]$ by Corollary \ref{cor:1}) and \eqref{mainthm.h2}, since
\[
	\int_{\Omega_\ep}|\partial_z v_\ep(y,z,t)|^2 \, dy \,dz \  
	= \  \ \int_\Omega
	\frac{ |\partial_z u_\ep(x,z,h_\ep^2 t)|^2 }\logeps\,dx \,dz \,
\]
and
$\int_{\Omega_\ep}\frac 12 |\nabla_x v_\ep(y,z,t)|^2 \, dy \le 
	G_\ep(u_\ep(\cdot, \cdot, h_\ep^2 t)) = G_\ep(u_\ep^0)$.
Also, 
\[
	\int_{\Omega_\ep}
	\ep_{ab}\partial_{az}\vp  \frac{\partial_b 
	\ve\cdot \partial_z \ve}\logeps \,dx\,dz
	\to -\pi \int_0^L \nabla^\perp\partial_z\vp(f_k^*(z),z,t)
	\cdot  \partial_zf^*_k(z,t) \, dz
\] 
for every $t$, due to \eqref{improve2}. It follows that
\[
	\int_0^\tau\!\!\int_{\Omega_\ep}
	\ep_{ab}\partial_{az}\vp  \frac{\partial_b 
	\ve\cdot \partial_z \ve}\logeps \,dx\,dz\,dt
	\to
	-\pi\int_0^\tau\!\! 
	\int_0^L \nabla^\perp\partial_z\vp(f_k^*(z),z,t)
	\cdot  \partial_zf^*_k(z,t) \, dz\,dt.
\]

{\bf 4}.
To describe the limit of the remaining term coming from \eqref{dtJzv.int}, first
note that \eqref{dac}, together with our assumptions on the support of $\vp$,
implies that 
\[
	\mbox{supp}(\ep_{ab}\partial_{ac}\vp \partial_b v_\ep 
	\cdot \partial_c v_\ep)(\cdot ,t)
	\subset \Omega_{\ep,k}(t)  
	:=
	\{ (x,z)\in \Omega_\ep : \ |x - f_k(z,t)| 
	\in [\frac {\rho_0}4 ,\frac{\rho_0}2] \} .
\]

Next, we follow standard arguments and write
\[
	\partial_b v_\ep \cdot \partial_c v_\ep 
	= \partial_c|v_\ep| \, \partial_c|v_\ep| + 
	\frac {j_b(v_\ep) j_c(v_\ep)}{|v_\ep|^2}.
\]
For the rest of this proof we will  write $j^*_\ep$ as an abbreviation for
$j^*_{\omega_\ep}(f^*)$, and $j^* := \lim_{\ep\to 0} j^*_\ep = j^*_{\R^2}(f^*)$.
With this notation, we further decompose the last term above as 
\begin{multline*}
	\frac {j_b(v_\ep) j_c(v_\ep)}{|v_\ep|^2}
	=
	j^*_{\ep,b} \  j^*_{\ep,c}+
	\left(\frac {j(v_\ep)}{|v_\ep|} - j^*_{\ep}\right)_b
	\left(\frac {j(v_\ep)}{|v_\ep|} - j^*_{\ep}\right)_c \\
	+j^*_{\ep,b}\left(\frac {j(v_\ep)}{|v_\ep|} - j^*_{\ep}\right)_c
	+j^*_{\ep,c}\left(\frac {j(v_\ep)}{|v_\ep|} - j^*_{\ep}\right)_b.
\end{multline*}
Thus,
\begin{align*}
	\int_0^\tau\!\!\int_{\Omega_\ep} \ep_{ab}\partial_{ac}\vp\, 
	& \partial_b v_\ep \cdot \partial_c v_\ep \,dz\,dz\,dt 
	\le \int_0^\tau\!\!\int_{ \Omega_{\ep,k}(t) }\ep_{ab}
	\partial_{ac}\vp \, j^*_{\ep,b}  j^*_{\ep,c}  \\
	& +\int_0^\tau\!\!\int_{ \Omega_{\ep,k}(t)  }\ep_{ab}\partial_{ac}\vp\,
	\left[ j^*_{\ep,b} \left(\frac {j(v_\ep)}{|v_\ep|} - j^*_{\ep}\right)_c
	+ j^*_{\ep,c}\left(\frac {j(v_\ep)}{|v_\ep|} - j^*_{\ep}\right)_b \right] \\
	& + \int_0^\tau\!\!\int_{\Omega_{\ep,k}(t) } |\nabla_x^2\vp|\,
	\left(|\nabla_x |v_\ep||^2 +
	\left|\frac {j(v_\ep)}{|v_\ep|} - j^*_{\ep}\right|^2\right).
\end{align*}
It follows from Proposition \ref{P:jstar} that the second term on the right-hand
side converges to $0$ as $\ep\to 0$. 

Using  Proposition \ref{prop:3} for a sequence $\delta_n\to 0$ and recalling
that $\Sigma_\ep(t)\to I_3(t)$ as $\ep \to 0$, we find that
\[
	\limsup_{\ep\to 0}\int_0^\tau\int_{\Omega_\ep}  
	|\nabla_x^2\vp|\left(|\nabla_x |v_\ep||^2 +
	\left|\frac {j(v_\ep)}{|v_\ep|} - j^*_{\ep} \right|^2\right) 
	\le C \int_0^\tau I_3(t)\,dt.
\]
Since $j_\ep^* \to j^*$ locally uniformly on $\R^2$, it is clear that
\[
	\int_0^\tau\int_{ \Omega_{\ep,k}(t) }\ep_{ab}
	\partial_{ac}\vp \, j^*_{\ep,b}  j^*_{\ep,c}  
	\to
	\int_0^\tau\int_{ \Omega_{\ep,k}(t) }\ep_{ab}
	\partial_{ac}\vp \, j^*_{b}  j^*_{c}  
\]
as $\ep \to 0$. Finally, we claim that
\[
	\int_0^\tau\int_{ \Omega_{\ep,k}(t) }\ep_{ab}
	\partial_{ac}\vp \, j^*_{b}  j^*_{c}  
	=
	\pi \int_0^\tau\int_0^L\nabla\vp(f_k^*,z,t))
	\cdot \nabla_k^\perp \calW(f^*(z,t)) \, dz\,dt.
\]
This is a small variant of a classical fact. We recall the proof for the
reader's convenience.  First note that for every $t$ and every $z\in (0,L)$, 
\[
	\int_{ \{ x\in \omega : |x-f_k(z,t)|\in[ \frac{\rho_0}4, \frac{\rho_0}{2}]\}}
	\ep_{ab}\partial_{ac}\vp \, j^*_{b}  j^*_{c}  \, dx
	=
	\lim_{s\to 0^+}
	\int_{\omega \setminus B_s(f_k(z,t))}\ep_{ab}
	\partial_{ac}\vp \, j^*_{b}  j^*_{c}  \,dx
\]
(where all integrands are evaluated at the fixed value of $t$). Indeed, the
right-hand side is independent of $s$ for $0<s<\rho_0/4$, since the integrand
vanishes identically in $B_{\rho_0/4}(f_k(z,t))$. For every $s<\rho_0/4$,
\begin{align}
	\int_{\omega \setminus B_s(f_k(z,t))}
	\ep_{ab}\partial_{ac}\vp \, j^*_{b}  j^*_{c}  \,dx
	&= 
	\int_{\omega \setminus B_s(f_k(z,t))}
	\ep_{ab}\partial_{ac}\vp \,( j^*_{b}  j^*_{c} 
	- \frac 12 \delta_{bc}|j^*|^2) \,dx \nonumber \\
	&= 
	-\int_{\partial B_{s}(f_k(z,t))} 
	\ep_{ab}\partial_a\vp \,( j^*_{b}  j^*_{c} 
	- \frac 12 \delta_{bc}|j^*|^2)\nu_c \nonumber \\
	&=  
	-\int_{\partial B_{s}(f_k(z,t))}(\nabla^\perp\vp \cdot j^*)(\nu \cdot j^*)
	- \frac 12 \nabla^\perp \vp\cdot\nu |j^*|^2.
	\label{jj.ibyp}
\end{align}
Note that
\[
	j^*(x,z,t) = \frac {(x- f_k(z,t))^\perp}{|x-f_k(z,t)|^2}
	+\widetilde j(x;k),\quad
	\mbox{ where } \widetilde j(x;k) = \sum_{\ell \ne k}
	\frac {(x- f_\ell (z,t))^\perp}{|x-f_\ell(z,t)|^2}.
\]
We decompose $j^*$ in this way on the right-hand side of \eqref{jj.ibyp}, then
expand and let $s$ tend to zero. This leads to 
\[
	\int_{ \{ x\in \omega : |x-f_k(z,t)|\in[ \frac{\rho_0}4, \frac{\rho_0}2]\}}
	\ep_{ab}\partial_{ac}\vp \, j^*_{b}  j^*_{c}  \, dx
	= -2\pi\nabla\vp(f_k^*(z,t),z,t)\cdot \widetilde j(f_k^*(z,t);k).
\]
Since
\[
	\nabla_k^\perp \calW(a) 
	:= -2\sum_{\ell\ne k} \frac{(a_k-a_\ell)^\perp}{|a_k-a_\ell|^2} = 
	-2\widetilde j(f_k^*(z,t);k), 
\]
this implies the claim, and the proof of Lemma \ref{lem:ptl} is completed.
\end{proof}

\begin{proof}[Proof of Lemma \ref{lem:pmt2}]
We apply Lemma \ref{lem:ptl} with
\[
	\vp(x,z,t)  = \chi_{\rho_0/4}(|x-f_k(z,t)|)\; |x - f_k(z,t)|^2,
\]
the bounds \eqref{helpful} being consequences of our assumptions on $\rho_0$
	and $\partial_t f$ in Theorem \ref{T.main}, and then sum the resulting 
inequalities over $k$. This leads to the estimate
\begin{align*}
	I_1(\tau) \le I_1(0)  &
	+ \int_0^\tau\int_0^L (f -f^*)\cdot \partial_t f 
	+ \partial_z f^\perp \cdot \partial_z f^* \ dz\, dt\\
	&-
	\int_0^\tau\int_0^L  (f-f^*)\cdot \nabla^\perp \calW(f^*) \ dz\, dt 
	+ C\int_0^\tau I_3(t)\,dt.
\end{align*}
The equation \eqref{IVF} satisfied by $f$ may be written
\[
	\partial_t f^\perp = \partial_{zz} f - \nabla \calW(f).
\]
Substituting this into the above inequality and integrating by parts, we obtain
\[
	I_1(\tau) \le I_1(0)  +
	\int_0^\tau\int_0^L  (f-f^*)\cdot(\nabla^\perp \calW(f) 
	-  \nabla^\perp \calW(f^*)) \ dz\, dt
	+ C\int_0^\tau I_3(t)\,dt.
\]
It follows from the definition of $t_0$ that
\[
	|\nabla^\perp \calW(f) -  \nabla^\perp \calW(f^*)| \le C |f-f^*|,
\]
and the conclusion follows immediately.
\end{proof}

\begin{proof}[Proof of Lemma \ref{lem:pmt3}]
We apply Lemma \ref{lem:ptl} with
\[
	\vp(x,z,t)  = \chi_{\rho_0/4}(|x-f_k(z,t)|)
	\left(-\partial_{zz}f_k(z,t) + \nabla_k \calW(f(z,t) \right) \cdot (f(z, t)
	- x)_k,
\]
the bounds \eqref{helpful} following once more from our assumtions in Theorem
\ref{T.main}, and then (implicitly) sum the resulting inequalities over $k$.
This leads to the estimate
\begin{align*}
	I_2(\tau) \le I_2(0)  &
	+ \int_0^\tau\int_0^L  \partial_t\left(-\partial_{zz}f_k 
	+ \nabla_k \calW(f )\right) \cdot (f - f^*)_k  \ dz\, dt\\
	&
	+ \int_0^\tau\int_0^L\partial_z \left(-\partial_{zz}f_k 
	+ \nabla_k \calW(f) \right)^\perp \cdot \partial_z f_k^* \ dz\, dt\\
	&+
	\int_0^\tau\int_0^L \left(\partial_{zz}f_k 
	- \nabla_k \calW(f) \right) \cdot \nabla_k^\perp \calW(f^*) \ dz\, dt + C\int_0^\tau I_3(t)\,dt.
\end{align*}
The middle integral on the right-hand side can be rewritten
\[
	\int_0^\tau\int_0^L\partial_z \partial_t f_k \cdot \partial_z f_k^* \ dz\, dt
	= 
	-\int_0^\tau\int_0^L\partial_{tzz}f_k \cdot  f_k^* \ dz\, dt,
\]
and hence cancels out part of the first integral. We then integrate by parts and expand
$\partial_t\nabla_k\calW(f)$ to obtain
\begin{align*}
	I_2(\tau)\le  I_2(0)  &
	+ \int_0^\tau\int_0^L  \partial_t f_j \cdot \partial_{zz}f_j
	\ dz\, dt\\
	&
	+ \int_0^\tau\int_0^L \partial_t f_j \cdot   \nabla_j\nabla_k\calW(f) 
	\cdot( f - f^*)_k  \ dz\, dt\\
	&+
	\int_0^\tau\int_0^L \left(\partial_{zz}f_k - \nabla_k \calW(f) \right) 
	\cdot \nabla_k^\perp \calW(f^*) \ dz\, dt + C\int_0^\tau I_3(t)\,dt  .
\end{align*}
Using the PDE \eqref{IVF} to eliminate $\partial_{zz}f$, we rewrite this as
\begin{align*}
	I_2(\tau) &\le I_2(0) + C\int_0^\tau I_3(t)\,dt\\
	&\qquad +\int_0^\tau\int_0^L \partial_t f_j\cdot\left[\nabla_j\calW(f^*) 
	- \nabla_j\calW(f) - \nabla_k\nabla_j\calW(f) \cdot (f^*-f)_k 
	\right] dz\,dt.
\end{align*}
Finally, it follows from the definition of $t_0$ that 
\[
\left|\nabla_j\calW(f^*) - \nabla_j\calW(f) - \nabla_k\nabla_j\calW(f) \cdot (f^*-f)_k 
\right| \le C|f^*-f|^2.
\]
The conclusion of the lemma follows immediately.
\end{proof}

\section{Proofs of variational results}\label{App}

In this section we present the proofs of Propositions  \ref{P:CJ}, \ref{P.improve}, 
and \ref{prop:3}. 

\subsection{Tools}

We start by assembling some tools that give information about the vortex structure of a function
satisfying \eqref{CJ.h1}, \eqref{CJ.h2} for small but fixed $\ep>0$, rather than in the limit
$\ep\to 0$. All of these  are established in \cite{CJ}, but in some cases our presentation here
differs a little. We therefore give short proofs that sketch  the arguments needed to obtain the
precise statements given here from those in \cite{CJ}. 

Our first result of this sort states that under assumptions \eqref{CJ.h1}, \eqref{CJ.h2}, for {\em
every} $z\in (0,L)$, if $\ep$ is small enough then  $u_\ep(\cdot ,z)$ has either $n$  distinct,
well-localized vortices clustered near the vertical axis, or a certain amount of ``extra energy". 
We will write
\[
e_\ep^{2d}(u)
:=
\frac 12 |\nabla_x\ue|^2  + \frac {1} {4\ep^2}(|\ue|^2-1)^2
\]
the Ginzburg-Landau energy density with respect to horizontal variables. 

\begin{lemma}
Assume that $u_\ep \in H^1(\Omega,\C)$ satisfies \eqref{CJ.h1} and \eqref{CJ.h2}.

There exist positive numbers $\theta, a,b, C$ and $\ep_0$ depending on $n, c_1,c_2$ such that $b<a$,
and if $0<\ep<\ep_0$, then for every $z\in (0,L)$ such that 
\begin{equation}\label{L3b.h2}
	\int_{\omega\times \{z\}} e^{2d}_\ep (u_\ep) \,dx \le  \pi(n+\theta)\logeps \ ,
\end{equation}
there exist $g^\ep_j(z)\in\R^2$ for $j = 1,\ldots, n$  such that 
\begin{equation}
	\| J_x u_\ep(\cdot, z) - \pi \sum_{j=1}^n \delta_{g^\ep_j(z)}\|_{F(\omega)} \le \ep^a \ ,
	\label{L3b.c1}
\end{equation}
\begin{equation}
	|g^\ep_j (z)- g^\ep_k(z)| \ge \ep^b
	\ \ \mbox{ for all }j\ne k, \ 
	\qquad\mbox{dist}(g^\ep_j(z), \partial \omega) \ge C^{-1} \ \ \mbox{ for all }j,
	\label{L3b.c2}
\end{equation}
\begin{equation}
	|g^\ep_j(z)| \le C h_\ep \ \  \mbox{ for all }j,
	\label{L3b.c2bis}
\end{equation}
\begin{equation}
	\int_{\omega\times \{z\}} e^{2d}_\ep(w) dx \ \ge \  
	n(\pi\vert \log\ep \vert +\gamma) +  W_{\omega}(g^\ep_1(z),\ldots, g^\ep_n(z))
	- C\ep^{(a-b)/2},  \ 
	\label{L3a.sharper}
\end{equation}
where $W_\omega$ is the renormalized energy defined in Section \ref{sec:furnot}. 
\label{L3b}\end{lemma}

\begin{proof}[Proof of Lemma \ref{L3b}, excluding estimate \eqref{L3b.c2bis}]
Given a sequence of functions  $u_\ep\in H^1(\Omega,\C)$ satisfying \eqref{CJ.h1} and \eqref{CJ.h2},
a set $\mathcal G^\ep_1 = \mathcal G^\ep_1(u_\ep)\subset (0,L)$ is defined in  equation (3.11) of
\cite{CJ}  with the following properties. First, if $z\not\in \mathcal G^\ep_1$ then 
\[
	\int_\omega e^{2d}_\ep (u_\ep)(x,z) \,dx \ge \ep^{-1/2},
\]
for all sufficiently small $\ep$ (where ``sufficiently small" may depend on the given sequence). And
second, if $z\in \mathcal G^\ep_1$ and \eqref{L3b.h2} holds, then there exist $g^\ep_j(z)\in
\omega$, for $j=1,\dots, n$, satisfying \eqref{L3b.c1},  \eqref{L3a.sharper} and  \eqref{L3b.c2}. 
These are proved in \cite{CJ}, Proposition 1 and Lemma 3 respectively, which actually assume a
somewhat weaker condition in place of \eqref{CJ.h1}.

The conclusions of the lemma, apart from \eqref{L3b.c2bis} (proved below), follow directly from these facts.
\end{proof}

We will henceforth write
\begin{equation}
	\calG(u_\ep) 
	:= \{ z\in (0,L):  \eqref{L3b.h2}  \mbox{ holds} \},
	\qquad
	\calB(u_\ep) := (0,L)\setminus \calG(u_\ep) \ .
	\label{calG.def}
\end{equation}
Thus, for every $z\in \calG(u_\ep)$, Lemma  \ref{L3b} provides a detailed description 
of the vorticity of $u_\ep(\cdot, z)$.

For $z\in \calG(u_\ep)$ we will write
\begin{equation}
	f^\ep_j(z) := g^\ep_j(z)/h_\ep.
	\label{fep.def}
\end{equation} 
Rescaling \eqref{L3b.c1}, we find that 
$\| J_x v_\ep(\cdot, z) - \pi \sum_{j=1}^n 
\delta_{f^\ep_j(z)}\|_{W^{-1,1}(\omega_\ep)} \le \ep^a/h_\ep $,
where $v_\ep(x,z) = u_\ep(h_\ep x,z)$ as usual.

\begin{remark}
It is clear from the proof in \cite{CJ} that $z\mapsto \chi_{\mathcal G(u_\ep)} g_j^\ep(z)$ may
be taken to be measurable.
\end{remark}

We next collect some conclusions that follow rather
easily from Lemma \ref{L3b}.

\begin{lemma}\label{C:GB}
Assume that $0<\ep < 1/2$ and that  $u_\ep \in H^1(\Omega,\C)$ satisfies 
\eqref{CJ.h1} and \eqref{CJ.h2}.
Then there exists a positive constant $C = C(c_1,c_2,n)$ such that
\begin{align}
	\label{uptoconstant}
	&\int_\Omega e^{2d}_\ep(u_\ep)
	\geq n\pi L \logeps +\pi n (n-1)L |\log h_\ep|- C,\\
	\label{calB.est1}
	&|\calB(u_\ep)|
	\le C\logeps^{-1},\\
	&\int_{z\in \calB(u_\ep)}\int_\omega e_\ep^{2d}(u_\ep) \, dx\,dz \le C,
	\label{calB.est2} 
	\\
	\label{dzu2.b}
	&\int_\Omega  |\partial_z u_\ep|^2 dx\;dz \leq C.
\end{align}
\end{lemma}

We will later improve on some of these estimates under the hypotheses of our main theorem.

\begin{proof}[Proof of Lemma \ref{C:GB}]
Conclusions \eqref{uptoconstant} and \eqref{dzu2.b} are proved in Lemma 9 of \cite{CJ}.
The proof relies on the parts of Lemma \ref{L3b} proved above, together with properties of the
renormalized energy $W_\omega$ (see Lemma 4 of \cite{CJ}) and a short argument using Jensen's
inequality. The proof also easily yields the other conclusions \eqref{calB.est1}, \eqref{calB.est2}
stated here. Indeed, the proof of Lemma 9 in \cite{CJ} actually shows\footnote{Note that the sets
$\calG^\ep_2$ and $\calB^\ep_2$ from \cite{CJ} coincide exactly with our sets $\calG(u_\ep)$ and
$\calB(u_\ep)$; compare our definitions \eqref{calG.def} with \cite{CJ}, equation (3.16).} that 
\[
	\int_{z\in \calG(u_\ep)}\int_\omega e^{2d}_\ep(u_\ep)\, dx\,dz \ge
	\Big(n \pi \logeps + n(n-1)\pi (|\log h_\ep| - C\Big)|\calG(u_\ep)|.
\]
On the other hand it is clear from the definitions that
\[
	\int_{z\in \calB(u_\ep)}\int_\omega e^{2d}_\ep(u_\ep)\, dx\,dz \ge
	(n\pi +\theta)\logeps \, |\calB(u_\ep)|.
\]
Since $e_\ep(u_\ep) = e_\ep^{2d}(u_\ep) + \frac 12 |\partial_z u_\ep|^2$ and $
|\calG(u_\ep)| + |\calB(u_\ep)|=L$, by comparing these estimates with the hypothesis \eqref{CJ.h2},
we easily obtain  \eqref{calB.est1} and \eqref{calB.est2}.
\end{proof}

We now state a result that establishes a sort of approximate equicontinuity 
of the map $z\in \calG(u_\ep)\mapsto \pi \sum \delta_{f^\ep_j(z)}$ for finite $\ep>0$.

\begin{lemma}Assume that \eqref{CJ.h1}, \eqref{CJ.h2} hold.
Then for every $\delta>0$, there exists positive constants $\ep_0, C$ such that if $0<\ep <\ep_0$,
then the following holds: 

Assume that $z_1,z_2$ are points in $\calG(u_\ep)$ such that $|z_1 - z_2| > \delta$, and let
$g_j^\ep(z_\ell)$ denote the points provided by Lemma \ref{L3b} for $\ell = 1,2$. 
Then for
$f^\ep_j(z_\ell) := g_j^\ep(z_\ell)/h_\ep$, 
\begin{equation}\label{ep.holder}
	\pi \min_{\sigma\in S_n} \sum_{j=1}^n \frac {|f_{\sigma(j)}^\ep(z_2) - f_j^\ep(z_1)|^2}{|z_2-z_1|}
	\le C.
\end{equation}
\label{L.ep.holder}
\end{lemma}

\begin{proof}[Proof of conclusion \eqref{L3b.c2bis} of Lemma \ref{L3b} and of Lemma \ref{L.ep.holder}]
Estimate \eqref{L3b.c2bis} is shown to hold in Step 3 of the proof of Lemma 12 in \cite{CJ}, via a
compactness argument based on Lemma \ref{labels}, see below.

Lemma \ref{L.ep.holder} then follows from Lemma \ref{labels} by almost exactly the same compactness
argument.  The constant $C$ appearing in \eqref{ep.holder} may be chosen to be a multiple of the
uniform bound for $\int_\Omega |\partial_z u_\ep|^2$, established in Lemma \ref{C:GB} and depending
only on $c_1,c_2$ from \eqref{CJ.h1}, \eqref{CJ.h2}. 
\end{proof}

The last result in this section is the lemma used in the compactness arguments described above. It
will be used again in the proof of  Proposition \ref{prop:3}. In \cite{CJ} it provides the basic estimate that
eventually implies that $z\mapsto f(z) = (f_1(z),\ldots, f_n(z))$ belongs to $H^1((0,L), (\R^2)^n)$,
see Proposition \ref{P:CJ}. 

\begin{lemma}\label{labels}
Assume that $(u_\ep)$ satisfies  \eqref{CJ.h1}, \eqref{CJ.h2}.
Let $v_\ep(x,z) := u_\ep(h_\ep x,z)$.

Assume that  $\{z^\ep_1\}$ and $\{z^\ep_2\}$ are sequences  in $[0,L]$ such that $z^\ep_j\rightarrow
z_j$ for $j=1,2$, with $0\le z_1 < z_2 \le L$, and that the following conditions hold for $j=1,2$
(perhaps after passing to a subsequence):
\begin{eqnarray*}
	&&J_x v_\ep(\cdot, z^\ep_j)  \to \pi \sum_{i=1}^{n(z)} \delta_{p_i(z_j)} 
	\qquad
	\mbox{ in }W^{-1,1}(B(R)), \ \ \mbox{ for all }R>0,
\end{eqnarray*}
(for certain  points $\{ p_i(z_j)\}_{i=1}^{n(z_j)}$, not necessarily distinct) and
\begin{equation*}
	\limsup_{\ep\to 0} \logeps^{-1}
	\int_{\omega} e^{2d}_\ep(u_\ep(x,z^\ep_j))dx \le M\pi 
\end{equation*}
for some $M >0$. Then $n(z_1) = n(z_2) =: m$, and 
\begin{equation*}
	\frac{\pi}{2}\min_{\sigma\in S_m}\sum_{i=1}^m\frac{|p_i(z_1)
	-p_{\sigma(i)}(z_2)|^2}{z_2-z_1}
	\ \le \  \liminf_{\ep \to 0} \int_{z_1}^{z_2} 
	\int_{\omega_\ep} \frac 12 |\partial_z u_\ep|^2\, dx\,dz .
\end{equation*}
\end{lemma}

\begin{proof}
This is essentially Lemma 10 of \cite{CJ}. Apart from some notational changes,
	the main difference is that Lemma 10 of \cite{CJ} is proved under an
	assumption that is somewhat weaker than \eqref{CJ.h1}. As a result, it is
	stated there for a rescaling $v_\ep(x,z) := u_\ep(\ell_\ep x,z)$ using a
	scaling factor $\ell_\ep$ that is shown only later to equal $h_\ep$.  With
	the stronger assumption \eqref{CJ.h1}, the proof can be simplified, and one
	can work directly with the $\ell_\ep = h_\ep$.
\end{proof}

\subsection{Proof of Proposition \ref{P:CJ}}\label{subsec:pp1}

\begin{proof}
With a couple of exceptions, everything in Proposition \ref{P:CJ} is taken
	directly from the statement of Theorem 3 in \cite{CJ}.

The first exception is the compactness assertion \eqref{comp}; in \cite{CJ},
	compactness is proved to hold only with respect to a weaker topology. To
	prove \eqref{comp}, we  argue as follows.  First note that
\begin{align}
	\int_{z\in \calB(u_\ep)}& \| J_x u_\ep(\cdot ,z) 
	- \pi\sum_{j=1}^n \delta_{h_\ep f_j(z)}\|_{W^{-1,1}(\omega)} dz 
	\nonumber \\
	&\qquad\qquad\le
	n\pi |\calB(u_\ep)|+ \int_{z\in \calB(u_\ep)} \| J_xu_\ep (\cdot,z)\|_{W^{-1,1}(\omega)} dz 
	\nonumber \\
	&\qquad\qquad\le
	n\pi |\calB(u_\ep)|+ C \logeps^{-1} \int_{z\in \calB(u_\ep)} e_\ep^{2d}(u_\ep)(x,z)dz
	\nonumber \\
	&\qquad\qquad\le
	C\logeps^{-1} = C h_\ep^2\label{gs33}
\end{align}
by standard Jacobian estimates (see for example \cite{JeSo} or
	\cite{SandSerf-book}) and Lemma \ref{C:GB}, for  $C = C(c_1,c_2, n)$.  On
	the other hand, by \eqref{L3b.c1} and \eqref{fep.def}, 
\begin{align}
	\int_{z\in \calG(u_\ep)}& \| J_x u_\ep(\cdot ,z) 
	- \pi\sum_{j=1}^n \delta_{h_\ep f_j(z)}\|_{W^{-1,1}(\omega)} dz 
	\nonumber \\
	&\le
	\int_{z\in \calG(u_\ep)} \| \pi \sum_{j=1}^n \delta_{h_\ep f_j^\ep(z)} 
	- \pi\sum_{j=1}^n \delta_{h_\ep f_j(z)}\|_{W^{-1,1}(\omega)} dz + C \ep^a.
	\label{gs44}
\end{align}
It is also shown in \cite{CJ}, Lemmas 13 and 14 that after passing to a suitable
	subsequence $\ep_k\to 0$, there is a set $H_G\subset (0,L)$ of full measure,
	such that if $z\in H_G$, then there exists $\ell = \ell(z)$ such that $z\in
	\calG(u_{\ep_k})$ for all $k\ge \ell$, and 
\begin{equation*}%\label{L13.14}
	\| \pi \sum_{j=1}^n \delta_{ f_j^{\ep_k}(z)} 
	- \pi\sum_{j=1}^n \delta_{f_j(z)}\|_{W^{-1,1}(B(R))}  \to 0 
	\qquad\mbox{ for every }R>0
\end{equation*}
as $k\to \infty$. This implies that 
\[
	\| \pi \sum_{j=1}^n \delta_{h_{\ep_k} f_j^{\ep_k}(z)} 
	- \pi\sum_{j=1}^n \delta_{h_{\ep_k} f_j(z)}\|_{W^{-1,1}(\omega)} = o(h_{\ep_k})
	\qquad\mbox{for every $z\in H_G$}
\]
as $k\to \infty$. It also follows from \eqref{L3b.c2bis} that 
\[
	\| \pi \sum_{j=1}^n \delta_{h_{\ep_k} f_j^{\ep_k}(z)} 
	- \pi\sum_{j=1}^n \delta_{h_{\ep_k} f_j(z)}\|_{W^{-1,1}(\omega)} \le C h_{\ep_k}
	\quad
	\mbox{for $z\in \calG(u_{\ep_k}) \setminus H_G$, }
\]
so the conclusion follows from the dominated convergence theorem, together with
	\eqref{gs33} and \eqref{gs44}.

The other assertion that is not taken directly from the statement of Theorem 3
	in \cite{CJ} is the estimate $\|f\|_{H^1}\le C(c_1,c_2)$. To prove this, we
	use \eqref{BCL}  to deduce that for $z\in H_G$, 
\begin{align*}
	\sum_i |f_i(z)| = \sum_i |f_i(z) - 0|&
	= \lim_{k\to\infty}\frac 1 {h_\ep} \| \sum_i \delta_{h_{\ep_k}f_i(z) }
	- n \pi\delta_0 \|_{W^{-1,1}(\omega)}\\
	& 
	= \lim_{k\to\infty}\frac 1 {\pi h_\ep} \|J_x u_\ep(\cdot, z)- 
	n\pi \delta_0 \|_{W^{-1,1}(\omega)} . 
\end{align*}
Thus Fatou's Lemma and \eqref{CJ.h1} imply that 
\[
	\| f \|_{L^1}\le C(c_1).
\]
We may then use Jensen's inequality and the fact from \cite{CJ} that  $G_0(f)\le
c_2$ to estimate
	\begin{align*}
	\frac \pi 2\int_0^L \sum_j |f_j'|^2 dz 
	&= G_0(f)  +  \pi \sum_{i\ne j}\int_0^L \log|f_i-f_j| dz \\
	&\le
	c_2 + L\pi  \sum_{i\ne j} \log \left( \frac 1 L \int_0^L |f_i- f_j| dz\right) \\
	&\le
	C(c_1,c_2).
\end{align*}
Finally, $\| f\|_{L^2}$ is controlled by interpolating between $\| f\|_{L^1}$
and $\|f'\|_{L^2}$.
\end{proof}

\subsection{Proof of Proposition \ref{P.improve}}\label{subsec:pp2}

\begin{proof}[Proof of \eqref{improve1}]
It suffices to show, given any subsequence satisfying
\eqref{CJ.h2}, \eqref{comp} for which
\[
	\logeps^{-1}\partial_a v_\ep \cdot \partial_b v_\ep\rightharpoonup 
	\mbox{ some limit, weakly as measures}
\]
that this limit can only equal $\pi\delta_{ab} \sum_i \delta_{f_i(z)}\otimes
	dz$.  For $z\in (0,L)$, let
\[
	E^{2d}_\ep(z)
	:=\frac 1\logeps \int_{\omega\times \{z\}} e_\ep^{2d}(u_\ep) \, dx = 
	\frac 1 \logeps \int_{ \omega_\ep\times\{z\}} e_{\ep'}^{2d}(v_\ep) \, dx 
\]
where $\ep'= \ep/h_\ep$.  It follows from the definition of $\calB(u_\ep)$ that
	$E_\ep^{2d}(z)\ge n\pi  +\theta$ for $z\in \calB(u_\ep)$, and  since
	\eqref{L3b.c2bis} implies that $W_\omega(g_1^\ep,\ldots, g_n^\ep)\ge n\pi
	|\log h_\ep|-C$, we deduce from \eqref{L3a.sharper} that $E_\ep^{2d}(z)\ge
	n\pi  - o(1)$ uniformly for  $z\in \calG(u_\ep)$, as $\ep \to 0$. On the
	other hand, the assumed energy scaling \eqref{CJ.h2} implies that $\int_0^L
	E_\ep^{2d}(z)\, dz\to n\pi L$ as $\ep \to 0$.  In view of these facts, after
	passing to a further subsequence if necessary, we may assume that 
\begin{equation}
	\frac 1 \logeps \int_{ \omega_\ep\times \{z\}} e_{\ep'}^{2d}(v_\ep) \, dx 
	\to n \pi \qquad\mbox{ for a.e. } z\in (0,L).
	\label{prop2.step1a}\end{equation}
Next, upon rescaling \eqref{comp} and  passing to a further subsequence,
\begin{equation}
	\|  J v_\ep - \pi \sum_{i=1}^n \delta_{f_i(z)}\|_{W^{-1,1}(\omega_\ep)} 
	\to 0 \qquad\mbox{ for a.e. $z\in (0,L)$.}
	\label{prop2.step1b}
\end{equation}
It follows from Theorem 5 in \cite{filaments} or Corollary 4 in
	\cite{SandSerf-product} that whenever the above two conditions hold ({\em
	i.e.} a.e.), 
\[
	\frac 1{\logeps} \partial_a v_\ep \cdot \partial_b v_\ep(\cdot, z) 
	\rightharpoonup \delta_{ab}
	\pi \sum_{i=1}^n \delta_{f_i(z)}\qquad\mbox{weakly as measures}.
\]
Now fix any $\phi\in C_c(\R^2\times [0,L])$, and let 
\[
	\Phi_\ep(z) :=  \frac 1 \logeps \int_{\omega_\ep\times  \{z\}} \phi(x,z) 
	\partial_a v_\ep \cdot \partial_b v_\ep(x, z) \, dx.
\]
We write $\Phi_\ep = \Phi_{\calG, \ep} + \Phi_{\calB,\ep}$, where 
$\Phi_{\calG, \ep} = \chi_{z\in \calG(u_\ep)}\Phi_\ep(z)$.
It follows immediately from \eqref{calB.est2} that
$ \Phi_{\calB,\ep}\to 0$ in $L^1((0,L))$. 
We may assume after passing to a subsequence that $\chi_{\calB(u_\ep)}\to 0$ a.e.. 
It then follows that
\[
	\Phi_{\calG,\ep}(z)\rightarrow \delta_{ab}\pi \sum_{i=1}^n 
	\phi( f_i(z),z)\qquad\mbox{ for a.e. }z.
\]
The definition of $\calG(u_\ep)$ implies that $\sup_z |\Phi_{\calG,\ep}(z)|\le
	(n\pi +\theta)\sup_{(x,z)}  |\phi(x,z)| \le C$. Thus the dominated
	convergence theorem implies that
\[
	\lim_{\ep\to 0 } 
	\int_0^L \Phi_\ep(z) 
	dz =  \delta_{ab}\pi \sum_{i=1}^n \phi( f_i(z),z)\, dz.
\]
This is \eqref{improve1}.
\end{proof}

\begin{proof}[Proof of \eqref{local.dz2.bound}]
For $\delta>0$, let 
\[
\mathcal I_\delta := \{ z\in (0,L) : \min_{i\ne j} |f_i(z)-f_j(z)| > \delta \}.
\]
We know from \eqref{fbounds} that $G_0(f)<\infty$, which implies that 
$|\mathcal I_\delta| \to L$ as $\delta\to 0$. It thus suffices to prove that for any nonnegative
$\phi\in C_c(\R^2\times[0,L])$ and for every $\delta>0$, 
\begin{equation}
	\liminf_{\ep\to 0}\int_{\omega_\ep \mathcal I_\delta } 
	\phi \frac{|\partial_z v_\ep|^2}\logeps \, dx\,dz \ge
	\pi \sum_{i=1}^n\int_{\mathcal I_\delta} |f_i'(z)|^2 \phi(f_i(z), z) \, dz . 
	\label{prop2.step2a}
\end{equation}
We may write $\mathcal I_\delta$ as a disjoint union of open intervals.
Let $I$ denote one such interval. In view of arguments in the proof of  \eqref{improve1}, it
suffices to prove that if  $f\in H^1(I, (\R^2)^n)$ is such that \eqref{prop2.step1a},
\eqref{prop2.step1b} hold for a.e. $z\in I$ and $\min_{z\in I}\min_{i\ne h} |f_i(z)-f_j(z)| \ge
\delta>0$, then \eqref{prop2.step2a} is satisfied (with $\mathcal I_\delta$ replaced by $I$). 

There are a number of proofs of this fact\footnote{These 
results assume that  \eqref{prop2.step1a}, \eqref{prop2.step1b} hold for {\em every} $z\in I$,
but the proofs extend to our situation with essentially no change.}
 when $\phi \equiv 1$; see for example \cite{wave} Proposition 3 or \cite{SandSerf-product},
Corollary 7. These proofs proceed by considering separately the energetic contributions associated
to each trajectory $z\mapsto (f_i(z),z)$, and they  show that for any $r>0$, and every $i\in
\{1,\ldots, n\}$, and every  interval  $J\subset I$ 
\[
	\liminf_{\ep\to 0}\int_{z\in J }\int_{B_r(f_i(z))} 
	\frac{|\partial_z v_\ep|^2}\logeps \, dx\,dz \ge
	\pi \int_{J} |f_i'(z)|^2  \, dz .
\]
This easily implies the desired estimate.
\end{proof}

\begin{proof}[Proof of \eqref{improve2}]
First, recalling that $v_\ep(x,z) = u_\ep(h_\ep x, z)$ and using \eqref{dzu2},  
\[
\int_\Omega |\partial_z u_\ep|^2\, dx\,dz = 
\int_\Omega \frac{ |\partial_z v_\ep|^2}\logeps dx\,dz \le C(c_1,c_2,n). 
\]
We may thus assume that $\logeps^{-1} \partial_z v_\ep \cdot \nabla_x v_\ep$
converges weakly to a limiting $\R^2$-valued measure, say $\lambda$ on $\R^2\times [0,L]$.

Now fix some $g\in C^1(( 0,L),  \R^2)$, and let
\[
	\tilde u_\ep(x,z) := u_\ep(x - h_\ep g(z), z), \qquad
	\tilde v_\ep(x,z) := \tilde u_\ep(h_\ep x, z) =  v_\ep(x - g(z), z).
\]
If we fix some $\tilde \omega \subset\subset \omega$ such that $0\in \tilde \omega$,
we may then take the domain of $\tilde u_\ep$ to be $\tilde \Omega := \tilde \omega\times(0,L)$,
for all sufficiently small $\ep$. (We remark that although we are ultimately interested in $u_\ep$ that is
periodic in the $z$ variable, here we do not assume that $g$ is periodic.)

It is straightforward to check from \eqref{comp} and the definition of $\tilde u_\ep$ that
\[
	\int_0^L \| J_x \tilde u_\ep(\cdot ,z) 
	- \pi\sum_{j=1}^n \delta_{h_\ep (f_j(z) + g(z))}  \|_{W^{-1,1}(\tilde\omega)} dz = o(h_\ep)
	\qquad
	\mbox{ as }\ep\to 0.
\]
Also, since $h_\ep = \logeps^{-1/2}$ and extending the definition \eqref{Gep.def} of $G_\eps$ to include a
dependence in the domain, we have 
\begin{align*}
G_\ep(\tilde u_\ep;\tilde \Omega)
& \le G_\ep(u_\ep;\Omega ) 
+ \int_\Omega
\frac{ |g'(z)\cdot\nabla_x u_\ep|}{\sqrt\logeps} |\partial_z u_\ep| 
+ \frac 12
\frac{ |g'(z)\cdot\nabla_x u_\ep|^2}\logeps \, dx\, dz \\
& \le c_2 + C\int_\Omega |\partial_z u_\ep|^2  + \frac{ |\nabla_x u_\ep|^2}\logeps
dx\,dz \\
&\le \tilde K_1
\end{align*}
for some suitable $\tilde K_1$, whenever $\ep$ is sufficiently small.
Thus \eqref{local.dz2.bound} implies that for any continuous $\tilde \phi \ge 0$,
\[
	\liminf_{\ep\to 0} \int \tilde\phi(x,z) 
	\frac{ |\partial_z \tilde v_\ep(x,z)|^2}\logeps\, dx\,dz\, 
	\ge  \sum_i\pi \int_0^L |\partial_z(f_i +g)(z)|^2\tilde\phi(f_i(z)+g (z),z)\, dz.
\]
Taking $\tilde \phi$ of the form $\tilde \phi(x,z) = \phi(x-g(z),z)$, we  get
	the more convenient expression
\[
	\liminf_{\ep \to 0} \int \phi(x-g(z),z) 
	\frac{ |\partial_z \tilde v_\ep(x,z)|^2}\logeps \, dx\,dz\,
 	\ge  \sum_i \pi \int_0^L |\partial_z(f_i +g)(z)|^2\phi(f_i(z),z)\, dz.
\]
On the other hand, by using the definition of $\tilde v_\ep$ and making the change of variables
$(x-g(x), z) \mapsto (x,z)$, we obtain 
\begin{align*}
	&\int \phi(x-g(z),z) |\partial_z \tilde v_\ep(x,z)|^2\, dx\,dz\, 
	=
	\int \phi(x,z) |\partial_z v_\ep(x,z)|^2\, dx\,dz\\
	&\hspace{2em}+\int \phi(x,z)\left( -2 g'(z)\cdot \nabla_x v_\ep(x,z) \cdot 
	\partial_z v_\ep(x,z) + |g'(z)\cdot \nabla_x v_\ep(x,z)|^2\right)dx\,dz.
\end{align*}
Dividing by $\logeps$, letting $\ep\to 0$, and invoking \eqref{dzu2} and
\eqref{improve1}, we find that 
\begin{align*}
	&
	\limsup_{\ep\to 0} \int \phi(x-g(z),z)
	\frac{ |\partial_z \tilde v_\ep(x,z)|^2}{\logeps}\, dx\,dz\, \\
	&\hspace{4em}
	\le C  - 2\int _{\R^2\times (0,L)}\!\!\phi(x,z) g'(z)\cdot d\lambda + 
	\sum_i \pi \int_0^L |\partial_z g(z)|^2\phi(f_i,z)\, dz.
\end{align*}
Combining this with the previous inequality and rewriting, we conclude that 
\[
	\int _{\R^2\times (0,L)} \phi(x,z)g'(z)\cdot d\lambda 
	+ \pi \int_0^L \phi(x,z)g'(z) \cdot 
	d\left( \sum_i f_i'(z)\delta_{f_i(z)}\otimes dz\right) \le C
\]
for $g, \phi$ as above, with $C$ depending on $c_1,c_2,f,n, \phi$ but
independent of $g$ . Since we may multiply a given $g$ by an arbitrary real
constant, it follows that in fact
\[
	\int \phi(x,z)g'(z) \cdot d\lambda   + \pi \int \phi(x,z)g'(z) \cdot 
	d\left( \sum_i f_i'(z)\delta_{f_i(z)}\otimes dz\right) = 0
\]
and hence that 
\[
	\lambda = - \pi \sum_i f_i'(z)\delta_{f_i(z)}\otimes dz.
\]
This is \eqref{improve2}.
\end{proof}

\subsection{Proof of Proposition \ref{prop:3}}\label{subsec:pp3}
Define 
\begin{equation*}
	\sigma^{2d}_\ep(z) 
	= 
	\sigma^{2d}_\ep(z; u_\ep, h_\ep f)
	=
	\int_{\omega} e^{2d}_\ep(u_\ep(x,z)) dx
	- W_\ep(h_\ep f(z) ; \omega)
	%\label{sigma2d.1}
\end{equation*}
where for $a\in \omega^n$, 
\begin{equation*}
	W_\ep(a;\omega) =n(\pi  \logeps   +\gamma) - \pi \sum_{i\ne j} \log|a_i(z) - a_j(z)| 
	+ \pi \sum_{i,j} H_\omega(a_i, a_j)\ .
	%\label{W.def}
\end{equation*}
Recall that $H_\omega$ is defined in Section \ref{sec:furnot}.  We interpret
$\sigma^{2d}_\ep(z)$ as the surplus 2d (horizontal) energy of $u_\ep$ at height
$z$, with respect to the vortex positions $h_\ep f(z)$. Further define 
\[
	\Sigma^{2d}_\ep = \Sigma^{2d}_\ep(u_\ep,  h_\ep f) 
	= \int_0^L \sigma^{2d}_\ep(z) dz.
\]

\begin{proof}[Proof of estimate \eqref{eq:galibier}]
Assume toward a contradiction that there exists a sequence $(u_\ep)_{\ep\in(0,1]}$ in
$H^1(\Omega,\C)$ such that  
\[
\int_0^L \| Ju_\ep(\cdot, z) - \pi \sum_{i=1}^n \delta_{h_\ep f_i(z)}\|_{W^{-1,1}(\omega)} dz
= o(h_\ep)
\]
and $G_\ep(u_\ep) - G_0(f) \le \Sigma_\ep \le 1$, but 
\begin{equation}\label{eq:reibilag}
\limsup_{\ep\to 0}
\int_0^L \int_{\omega \setminus \cup_{i=1}^n B(h_\ep f_i(x), h_\ep r)}
	e_\ep(|u_\ep|) + \frac14 \left|\frac  {ju_\ep}{|u_\ep|}  - j^*_{h_\ep f}\right|^2  - K_3\Sigma_\ep >0
\end{equation}
for $K_3$ to be chosen in a moment, and depending only on $\| f\|_{H^1}$ and  $r< \frac 14 \rho_f$.

This sequence satisfies the hypotheses \eqref{CJ.h1}, \eqref{CJ.h2} of Lemma \ref{L3b} with and
$c_1=1+  n\pi L \|f\|_\infty$ and $c_2 = G_0(f)+1$,  which are both controlled by $\| f\|_{H^1}$ and
$r$. Let $\theta = \theta(n,c_1,c_2)$  be the constant found in Lemma \ref{L3b}.  We will obtain a
contradiction to \eqref{eq:reibilag} with $K_3 = \frac 4\theta n\pi + 4$, thereby proving
\eqref{eq:galibier} for that value of $K_3$.

For this choice of $\theta$, we define sets $\calG(u_\ep)$ and $\calB(u_\ep)$ as in
\eqref{calG.def}. For $z\in \calG(u_\ep)$, Lemma \ref{L3b} provides points $g_j^\ep(z)$ satisfying
\eqref{L3b.c1}, \eqref{L3b.c2} for $0<\ep<\ep_0(n, \| f\|_{H^1}, \rho_f, \Sigma)$, with constants
such as $a$ in \eqref{L3b.c1} depending on the same quantities.

Setting $f^\ep_j(z) = h_\ep^{-1}g^\ep_j(z)$, it follows from  \eqref{L3b.c1} that
\begin{equation}\label{gfL1}
\int_{z\in \calG(u_\ep)}
 \|  \sum_{i=1}^n \delta_{h_\ep f^\ep_j(z)} - \sum_{i=1}^n \delta_{h_\ep f_i(z)}\|_{W^{-1,1}(\omega)} dz
= o(h_\ep) 
\quad  \mbox{ as }\ep\to 0.
\end{equation}
Our first goal is to strengthen this to read
\begin{equation}\label{gfunif}
\sup_{z\in \calG(u_\ep)}
 \|  \sum_{i=1}^n \delta_{h_\ep f^\ep_j(z)} - \sum_{i=1}^n \delta_{h_\ep f_i(z)}\|_{W^{-1,1}(\omega)}
= o(h_\ep) \quad  \mbox{ as }\ep\to 0.
\end{equation}
In brief, this follows from a compactness argument based on \eqref{gfL1} and Lemma \ref{labels}.
Here are the details: 

Assume toward a contradiction that \eqref{gfunif} fails. Then there exists a (sub)sequence $\ep\to
0$ and points $z_{\ep}\in \calG(u_{\ep})$ such that
\begin{equation}\label{L1lowerbd}
 \|  \sum_{i=1}^n \delta_{ h_\ep f^\ep_j(z_\ep)} - 
\sum_{i=1}^n \delta_{ h_\ep f_i(z_\ep)}\|_{W^{-1,1}(\omega)}
 \ge c  h_\ep  >0\quad \mbox{for all $\ep$. }
\end{equation}
It follows from \eqref{calB.est1} and \eqref{gfL1} that for all sufficiently small terms in the same
subsequence, we may find points $\zeta_\ep\in \calG(u_\ep)$ such that 
\[
	\|  \sum_{i=1}^n \delta_{ h_\ep f^\ep_j(\zeta_\ep)} 
	- \sum_{i=1}^n \delta_{ h_\ep f_i(\zeta_\ep)}\|_{W^{-1,1}(\omega)}
	= o(h_\ep), \quad\mbox{ and }\quad \alpha < | z_\ep - \zeta_\ep|< 2\alpha
\]
for some $\alpha$ to be fixed below. Extracting a further subsequence we may
	assume that $z_\ep \to z$ and $\zeta_\ep \to \zeta$, and that there exist
	$m\le n$ and $p_1,\ldots, p_m\in \R^2$ such that 
\[
	\sum_{i=1}^n \delta_{ f_i^\ep(\zeta_\ep)} \to 
	\sum_{i=1}^n \delta_{ f_i(\zeta)}, \quad
	\mbox{ and }\quad
	\sum_{i=1}^n \delta_{ f_i^\ep(z_\ep)} \to 
	\sum_{i=1}^m \delta_{ p_i(z)} 
\]
in $W^{-1,1}(B(R))$ for every $R>0$. (In fact both limits hold in stronger topologies as well.)
These facts and \eqref{L3b.c1} imply that for $v_\ep(x,z) := u_\ep(h_\ep x,z)$,
\[
J_x v(\cdot, \zeta_\ep) \to \pi \sum_{i=1}^n \delta_{ f_i(\zeta)} , \qquad
J_x v(\cdot, z_\ep) \to \pi \sum_{i=1}^m \delta_{ p_i(z)} 
\]
in the same topology. Then Lemma \ref{labels} and  conclusion \eqref{dzu2}
from Proposition \ref{P:CJ}  imply that $m=n$ and that 
\[
\min_{\sigma\in S_n} \sum_{i=1}^n |f_i(\zeta) - p_{\sigma(i)}(z)|^2 \le |z-\zeta| C \le 2\alpha C.
\]
(Here and below, the constant depends on $f$ and $\Sigma$.) 
On the other hand, since $f$ is H\"older continuous, it follows from \eqref{L1lowerbd} that
\[
\min_{\sigma\in S_n} \sum_{i=1}^n |f_i(\zeta)- p_{\sigma(i)}(z))| \ge
\min_{\sigma\in S_n} \sum_{i=1}^n |f_i(z)- p_{\sigma(i)}(z))|  -nC|z-\zeta|^{1/2} \ge
c - nC\alpha^{1/2}.
\]
A contradiction is reached by choosing $\alpha$ sufficiently small, depending only
on $f, \Sigma$, and $c$. This completes the proof of \eqref{gfunif}.

\medskip

Next, we remark that in view of the fact that $\rho_f>0$, it follows from \eqref{gfunif} and
\eqref{BCL} that the labels on $f^\ep_i$ may be chosen so that 
\begin{equation}\label{tatum}
\sup_{z\in \calG(u_\ep)} |f^\ep_i(z) - f_i(z)| \to 0 \quad\mbox{ as }\ep\to 0.
\end{equation}

We will write
\[
\omega(z, \ep, f) := \omega \setminus \cup_{i=1}^n B(h_\ep f_i(z), h_\ep r).
\]
For $z\in \calG(u_\ep)$, Theorem 2 of \cite{JeSp2}, for which the main hypothesis is a consequence
of \eqref{L3b.c1}, provides certain integral estimates on $\omega \setminus \cup_{i=1}^n B( h_\ep
f^\ep_i(z),C \ep^{a/2} )$, where $a>0$ comes from  \eqref{L3b.c1} and $C$ depends on various
ingredients that are fixed. It follows from \eqref{gfunif} and \eqref{BCL} that if $\ep$ is
sufficiently small, then for every $z\in \calG(u_\ep)$, this set contains $\omega(z, \ep, f)$.
Theorem 2 of \cite{JeSp2} thus implies that for every $z\in \calG(u_\ep)$, 
\begin{align*}
&\int_{ \omega(z, \ep, f) \times\{z\}} 
e_\ep^{2d}(|u_\ep|) + \frac14 \left|\frac  {ju_\ep}{|u_\ep|}  - j^*_\omega(h_\ep f^\ep(z))\right|^2 \, dx\\
&\qquad\qquad	\le \int_{\omega\times \{z\}}
e_\ep^{2d}(u)\, dx - \left[ n(\pi \logeps + \gamma) + W_\omega(h_\ep f^\ep(z)) \right]
+ C \ep^{a/2}.
\end{align*}
We recall that $W_\omega$ is defined in Section \ref{sec:furnot}. It is easy to check from the
definition there that 
\[
n(\pi \logeps + \gamma) + W_\omega(h_\ep f^\ep(z))
= \pi \calW(f^\ep(z))+ \kappa(n,\ep ,\omega) + O(h_\ep)
\]
where $\calW$ is introduced in \eqref{calW.def}. Thus
\[
\begin{aligned}
&\int_{z\in \calG(u_\ep)}
\int_{\omega(z, \ep, f) \times\{z\}} 
e_\ep^{2d}(|u_\ep|) + \frac18 \left|\frac  {ju_\ep}{|u_\ep|}  
- j^*_\omega(h_\ep f(z))\right|^2 \, dx\,dz
\\
&\qquad\qquad \le
\int_{z\in \calG(u_\ep)} 
\int_{\omega(z, \ep, f) \times\{z\}}   \frac 14
\left| j^*_\omega(h_\ep f(z))-  j^*_\omega(h_\ep f^\ep(z))\right|^2 \, dx\,dz
\\
&\qquad\qquad 
+
\int_{z\in \calG(u_\ep)} 
\left(\int_{\omega \times \{z\} } e^{2d}_\ep(u)\, dx - \kappa(n,\ep,\omega) -  
\pi \calW(f^\ep(z))\right) dz +O(h_\ep).
\end{aligned}
\]
It follows from  \eqref{tatum} and Lemma \ref{lem:jstardiff} below that the first term on the
right-hand side vanishes as $\ep \to 0$. Using this, we add and subtract various terms to rewrite
the above inequality as 
\begin{equation}\label{monk}
\begin{aligned}
	&\int_{z\in \calG(u_\ep)}
	\int_{\omega(z, \ep, f) \times\{z\}} 
	e_\ep^{2d}(|u_\ep|) + \frac18 \left|\frac  {ju_\ep}{|u_\ep|}  
	- j^*_\omega(h_\ep f(z))\right|^2 \, dx\,dz
	\\
	&\qquad\qquad \le G_\ep(u_\ep) - G_0(f) -
	\left( \int_\Omega \frac {| \partial_z u_\ep|^2}2\, dx\,dz 
	- \frac \pi 2\int_0^L |f'(z)|^2\,dz \right)
	\\
	&\qquad\qquad \qquad
	-
	\int_{z\in \calB(u_\ep)} 
	\left(\int_{\omega \times \{z\} } e^{2d}_\ep(u)\, dx - 
	\kappa(n,\ep,\omega) - \pi \calW(f(z))\right) \, dz + o(1).
\end{aligned}
\end{equation}
Clearly $|\calW(f)|$ is bounded by a constant depending on $n, \rho_0$ and $\|f\|_{H^1}$,
and it follows that $ \kappa(n,\ep,\omega) + \pi \calW(f(z)) \le (\pi n + \frac \theta 2)\logeps$
for all  sufficiently small $\ep$. Then the definition of $\calB(u_\ep)$ implies that 
$\int_{\omega \times \{z\} } e^{2d}_\ep(u)\, dx - \kappa(n,\ep,\omega) 
-\pi\calW(f(z)) \ge \frac  \theta 2\logeps$ when $z\in \calB(u_\ep)$.
Taking $\ep$ smaller, if necessary, we may assume by \eqref{local.dz2.bound} that 
\[
	\int_\Omega \frac {| \partial_z u_\ep|^2}2\, dx\,dz 
	- \frac \pi 2\int_0^L |f'(z)|^2\,dz \ge  - 
	\varpi\delta
\]
for $\varpi>0$ to be chosen. Employing this in  \eqref{monk} and discarding the left-hand side, we
deduce that 
\[
	|\calB(u_\ep)|  \le \frac 4\theta (\Sigma_\ep +\varpi\delta)  \logeps^{-1}
\]
for all sufficiently small $\ep>0$. Returning to \eqref{monk} with this new information, we deduce
that 
\begin{align*}
	\int_{z\in \calB(u_\ep)}  \int_{\omega \times \{z\} } e^{2d}_\ep(u)\, dx\, dz
	&\le \Sigma_\ep + \varpi \delta  
	+ \frac 4\theta (\Sigma_\ep +\varpi\delta)(n\pi+\frac\theta 2) \\
	& \le (3+\frac {4n\pi}\theta)\Sigma_\ep  +\frac \delta 4 + o(1)
\end{align*}
provided $\varpi\le \frac 14 $ is chosen small enough, depending only on $n$ and $\theta$, which
itself is universal. Then, since 
\[
	e_\ep^{2d}(|u_\ep|)+\frac18 \left|\frac  {ju_\ep}{|u_\ep|}  - j^*_\omega(h_\ep f(z))\right|^2
	\le
	e^{2d}_\ep(u) + \frac 14|j^*_\omega(h_\ep f(z))|^2,
\]
we use \eqref{monk} and the above estimate of $|\calB(u_\ep)|$ to find that
\[
\begin{aligned}
	&\int_0^L
	\int_{\omega(z, \ep, f) \times\{z\}} 
	e_\ep^{2d}(|u_\ep|) + \frac18 \left|\frac  {ju_\ep}{|u_\ep|}  
	- j^*_\omega(h_\ep f(z))\right|^2 \, dx\,dz
	\\
	\qquad \qquad&\le  (4+\frac {4n\pi}\theta)\Sigma_\ep  +\frac \delta 2
	+
	\int_{z\in \calB(u_\ep)} 
	\int_{\omega(z,\ep, f) \times \{z\} }  
	\frac 14|j^*_\omega(h_\ep f(z))|^2 dx  \, dz + o(1).
	\end{aligned}
\]
Finally,
\[
	\int_{\omega(z,\ep, f) \times \{z\} }  
	\frac 14|j^*_\omega(h_\ep f(z))|^2 dx  \, dz
	\le  C|\log h_\ep| = o(\logeps)
\]
for a constant that depends only on $n$ and  $\|f\|_{H^1}$ and $r$; 
this can be verified by arguments similar to those in Lemma \ref{lem:jstardiff} below.
Using this in the above inequality, we conclude that
\[
	\int_0^L
	\int_{\omega(z, \ep, f) \times\{z\}} 
	e_\ep^{2d}(|u_\ep|) + \frac18 \left|\frac  {ju_\ep}{|u_\ep|}  
	- j^*_\omega(h_\ep f(z))\right|^2 \, dx\,dz 
	\le \left(\frac 4 \theta n\pi + 4\right) \Sigma_\ep + \frac 34 \delta 
\]
for all sufficiently small $\ep$.  This contradicts \eqref{eq:reibilag} and
completes the proof of \eqref{eq:galibier}.
\end{proof}

Note that one can repeat the above proof with essentially no change, after
replacing $f$ in \eqref{eq:reibilag} and the two preceding assumptions by a
sequence $\tilde f^\ep$  with a uniform upper bound on $\|\tilde f^\ep \|_{H^1}$
and the uniform lower bound on $\rho_{ \tilde f^\ep}\ge 4r$, for $r$ fixed. Then
essentially\footnote{after extracting a uniformly convergent subsequence of $\{
	\tilde f^\ep\}$ } the same argument as above leads to the same
contradiction, establishing \eqref{eq:galibier} with $\ep_3,c_3$ that depend
only on $\|f \|_{H^1}$ and $r$. 

Next is the lemma that was used above.

\begin{lemma}\label{lem:jstardiff}
Assume that $a,a'\in \omega^n$ and that there exist $r_0 \ge r_1>0$ such that
\[
	\dist(a_i,\partial \omega)>r_0\qquad \mbox{ and }
	\qquad |a_i-\tilde a_i|\le \frac 12 r_1\le \frac 1{4}\rho_a
	\quad\mbox{ for all }i.
\]
Then
\[
	\int_{\omega\setminus \cup B_{r_1}(a_i)} |j^*_\omega(a)  
	- j^*_\omega(a')|^2\, dx \le
	C(n,r_0,\omega)|a-a'|^2 + C(n)(\frac{ |a-a'|}{r_1})^2.
\]
In particular, the above constants are independent of $r_1$.
\end{lemma}

\begin{proof}Using notation from Section \ref{sec:furnot}, 
\begin{align*}
	|j^*_\omega(x;a) - j^*_\omega(x;a') |^2 
	&\le 2n \sum_i \left| \frac{x-a_i}{|x-a_i|^2} - \frac{x-a_i'}{|x-a_i'|^2}\right|^2\\
	&\qquad  + 2n\sum_i |\nabla H_\omega(x,a_i) - \nabla H_\omega(x, a_i')|^2.
\end{align*}
The definition of $H_\omega$ and the maximum principle imply that 
\[ 
	|\nabla H_\omega(x,a_i) - \nabla H_\omega(x, a_i')| \le C(r_0)|a_i-a_i'|,
\]
and a short computation shows that if $|x-a|\ge 2 |a-a'|$, then
\[
	\left| \frac{x-a_i}{|x-a_i|^2} - 
	\frac{x-a_i'}{|x-a_i'|^2}\right|^2 \le 4 \frac{|a_i-a_i'|^2}{|x-a_i|^4}.
\] 
Thus
\begin{align*}
	&\int _{\omega\setminus \cup B_{r_1}(a_i)}
	|j^*_\omega(x;a) - j^*_\omega(x;a') |^2 \\
	&\qquad\qquad\le 
	2n|a-a'|^2 \int_{\R^2\setminus B_{r_1}(0)} |x|^{-4}dx
	+ C(n,r_0,\omega)|a-a'|^2 
\end{align*}
from which the conclusion of the lemma follows.
\end{proof}

\begin{proof}[Proof of \eqref{eq:lautaret}]

Assume toward a contradiction that there is a subsequence along which
\eqref{lessthanc2}, \eqref{eq:nakaayi} and \eqref{eq:izoard} hold for every
$\ep$, but there exists $\eta_1>0$ such that
\begin{equation}
	\begin{aligned}
	\lim_{\ep\to 0}
	h_\ep^{-1}\int_0^L \|J_x u_\ep(\cdot,z) - 
		\pi\sum_{i=1}^h  \delta_{h_\ep f_i(z)}\|_{F(\omega)} dz 
	&\ge  
	\lim_{\ep\to 0}
	\left(\pi n L (T^f_{r,\ep}(u_\ep)+\eta_1)\right)^\frac12 \\
 	& =: (\pi n L(T_{lim}+\eta_1))^{1/2}.
	\end{aligned}
	\label{gs1}
\end{equation}
Clearly \eqref{lessthanc2}, \eqref{eq:nakaayi} imply that the hypotheses of
	Proposition \ref{P:CJ} are satisfied (with a larger constant in
	\eqref{CJ.h1} than in \eqref{lessthanc2}), so we may use the proposition to
	find a subsequence, still denoted $(u_\ep)$, and a function $f^0\in
	H^1((0,L),(\R^2)^n)$ such that 
\begin{equation}\label{gs1a}
	\int_0^L \| J_x u_\ep(\cdot ,z) - \pi\sum_{j=1}^n 
	\delta_{h_\ep f_j^0(z)}\|_{W^{-1,1}(\omega)} dz = o(h_\ep)
\end{equation}
as $\ep\to 0$. 

We will first show that, after choosing $c_3$ suitably small and possibly
	relabelling,
\begin{equation}
	\| f_j - f^0_j\|_{L^\infty((0,L))} \le r \quad\mbox{ for }j=1,\ldots, n.
\label{ff-f0}\end{equation}
We start by noting from \eqref{lessthanc2}, \eqref{gs1}, and \eqref{gs1a} that 
\[
	\left(\pi nL (T_{lim} + \eta_1)\right)^\frac12  \le 
	\lim_{\ep \to 0}\frac 1 h_\ep \int_0^L \| \pi \sum_j (\delta_{h_\ep f_j(z)} - 
	\delta_{h_\ep f_j^0(z)}) \|_{W^{-1,1}(\omega)} \le  c_3.
\]
It follows from \eqref{BCL} that for all sufficiently small $\ep$ and all $z$,
\[
	\| \pi \sum_j (\delta_{h_\ep f_j^0(z)} - 
	\delta_{h_\ep f_j(z)}) \|_{W^{-1,1}(\omega)}
	=
	\pi h_\ep \min_{\sigma\in S_n} \sum_j 
	|f_j(z) - f_{\sigma(j)}^0(z)|.
\]
Thus
\begin{equation}\label{f-f0.L1}
  \left(\pi nL T_{lim}\right)^\frac12 +\eta_1 \le 
	\pi \int_0^L \min_{\sigma\in S_n} \sum_j 
	|f_j(z) - f_{\sigma(j)}^0(z)| \ dz \le c_3.
\end{equation}
In particular, this implies that
\[
	\| f\|_{L^1}  \le C(f^0, c_3).
\]
It follows from a Sobolev embedding and \eqref{fbounds} that there exists $C =
	C(f^0,c_2,c_3)$ such that 
\begin{equation}
	[ f ]_{C^{0,1/2}} \le \| f'\|_{L^2} \le C,\qquad
	\mbox{ and thus }
	[ f - f^0 ]_{C^{0,1/2}} \le C . 
	\label{f.holder}
\end{equation}
Next, we deduce from \eqref{f-f0.L1} and Chebyshev's inequality that
\[
	\left|
	\left\{ z\in (0,L) : 
	\min_{\sigma\in S_n} \sum_j  |f_j(z) - f_{\sigma(j)}^0(z)|  > r/2
	\right\} \right| \le \frac {2c_3}r.
\]
If $\min_{\sigma\in S_n} \sum_j  |f_j(z_0) - f_{\sigma(j)}^0(z_0)|>r$ for any
$z_0\in (0,L)$, then it follows from \eqref{f.holder} that 
\[
	\min_{\sigma\in S_n} \sum_j  |f_j(z) - f_{\sigma(j)}^0(z)|   >r/2
	\quad\mbox{ for all }z\in (0,L)\mbox{ such that }|z-z_0| < r^2/C.
\]
Fixing $c_3$ small enough (which only decreases the constant $C(f_0,c_2, c_3)$
in \eqref{f.holder}), we can arrange that the two above estimates are
incompatible. (This adjustment to $c_3$ again depends only on $\rho_f\ge 4r$ and
$\|f\|_{H^1}$.) It follows that for this choice of $c_3$, 
\[
	\min_{\sigma\in S_n} \sum_j  |f_j(z) - f_{\sigma(j)}^0(z)| 
	\le r \qquad\mbox{ for every } z\in (0,L).
\]
As a result,  we can find a single permutation $\pi$, independent of $z$, such
that $\sum_j |f_j(z) - f_{\pi(j)}^0(z)| = \min_\sigma \sum_j  |f_j(z) -
f_{\sigma(j)}^0(z)| \le r$ for all $z$. Using this permutation $\pi$ to relabel
the indices, we obtain \eqref{ff-f0}. 

If we write $\varphi(x) :=  \chi_r(\frac{|x|}{h_\ep})(\frac{|x|}{h_\ep})^2$,
then since $\|\nabla_x \varphi \|_\infty \le C/h_\ep$, it follows from
\eqref{lessthanc2},  \eqref{gs1a} that 
\begin{align*}
	T_{lim} & = \pi \sum_{i,j} \int_0^L  
	\chi_r(|f_j(z) - f^0_i (z)|) |f_j(z) - f^0_i(z)|^2\, dz.
\end{align*}
However, since $|f^0_i - f^0_j|\ge 4r$, we see from \eqref{ff-f0} that
\[
	\chi_r(|f_j(z) - f^0_i(z) |) = 
	\delta_{ij}\quad\mbox { for all $i,j$ and all $z\in(0,L)$}.
\]
So we obtain
\[
	\pi \| f - f^0\|_{L^2}^2  = T_{lim}.
\]
On the other hand, since we have by now arranged that 
\[
	\min_\sigma \sum_j  |f_j(z) - f_{\sigma(j)}^0(z)| = 
	\sum_j  |f_j(z) - f_j^0(z)| 
	\le \sqrt n |f(z)-f^0(z)|\quad\mbox{ for all $z$,}
\]
we pass to the limit in \eqref{gs1} to find that 
\[
	\sqrt{n\pi L} (\sqrt \pi \| f - f^0\|_{L^2}  +\eta_1) \le 
\sqrt{n}    \pi \| f - f^0\|_{L^1},
\]
in contradiction to the Cauchy-Schwarz inequality. Thus \eqref{eq:lautaret}
holds.
\end{proof}

\section{Compactness in time}\label{sec:compact}

In this last section we present the proofs of Proposition \ref{prop:4},
Corollary \ref{cor:1} and Proposition \ref{P:jstar}.

\subsection{Proof of Proposition \ref{prop:4}}
\begin{proof} We only need to prove \eqref{eq:bardet1}, since all other
	conclusions follow from that and Proposition \ref{prop:3}. 

To prove \eqref{eq:bardet1}, we define the stopping time
\[
	t^* := \sup \{ t>0 : u_\ep(\cdot, \cdot, h_\ep^2 s) 
	\mbox{ satisfies } \eqref{lessthanc2}, 
	\eqref{eq:izoard}\mbox{ for all }s\in (0,t)\}
\]
where $f$ should be replaced by $g$ in \eqref{lessthanc2}, \eqref{eq:izoard}. By
	a change of variables, 
\[
	T^g_{r,\ep}(u_\ep(\cdot, \cdot, h_\ep^2 t)) = 
	T^g_{r}(v_\ep(\cdot, \cdot, t)) ,
\]
where $T^g_r := T^g_{r,1}$ and $u_\ep, v_\ep$ are related by
	\eqref{rescale.witht}.  We use \eqref{dtJzv.int} with $\vp(x,z,t) =
	\chi^g_r(x,z)$ to find that
\[
	\frac d{dt} T^g_r(v_\ep(\cdot, \cdot, t)) \le 
	\left|
	\int \ep_{ab}\partial_{ac} \chi^g_r  \  \partial_b\ve \cdot \partial_c\ve \,dx\,dz
	\right|
	+
	\left|
	\int \ep_{ab}\partial_{az} \chi^g_r  \frac{\partial_b \ve\cdot \partial_z \ve}\logeps \,dx\,dz
	\right|.
\]
The definition of $\chi^g_r$ implies that $\partial_{ac}\chi^g_r(x,z) = 2 \delta_{ac}$
when $|x-g_i(z)|<r$ for some $i$, and hence that 
\[
	\ep_{ab}\partial_{ac} \chi^g_r  \  \partial_b\ve \cdot \partial_c\ve  = 0\mbox{ in }\cup_i
	B(g_i(z),r).
\]
In addition,
\[
	|\nabla_x v_\ep|^2 \le
	2 e_\ep(|v_\ep|) + \frac {|j(v_\ep)|^2 } {|v_\ep|^2} 
	\le
	2 e_\ep(|v_\ep|) + 2 \left|\frac  {ju_\ep}{|u_\ep|}  - j^*_\omega(h_\ep f)\right|^2
	+
	2\left| j^*_\omega(h_\ep f)\right|^2.
\]
The definition of $t^*$ allows us to apply estimates from Proposition \ref{P:CJ}
	(with $c_1 = c_4 + n\pi L \|g\|_\infty$ and $c_2 = G_0(g)+1$)  and
	Proposition \ref{prop:3} (with $\delta = \Sigma  = 1$ for example) to
	$v_\ep(\cdot,\cdot, t)$, for any $t\in (0, t^*)$, as long as $c_4, \ep_4$
	are taken to be small enough, depending only on $\| g \|_{H^1}, n$ and $r$.
	We may therefore deduce from  \eqref{eq:galibier} that 
\[
	\left|
	\int \ep_{ab}\partial_{ac} \chi^g_r  \  \partial_b\ve \cdot \partial_c\ve \,dx\,dz
	\right|
	\le C(K_3 + 1 ) \| \nabla^2_x \chi^g_r\|_\infty  = C(r, n,g).
\]
The remaining integral on the right-hand side is estimated by using \eqref{dzu2}
	(which after rescaling to $v_\ep$ acquires a factor of $\logeps^{-1}$) to
	find that
\begin{align*}
	\left|
	\int \ep_{ab}\partial_{az} \chi^g_r  
	\frac{\partial_b \ve\cdot \partial_z \ve}\logeps \,dx\,dz
	\right|
	&\le \frac1 \logeps \| \nabla_x \partial_z \chi^g_r\|_{L\infty} 
	\, \|\nabla_x v_\ep\|_{L^2}
  	\|\nabla_z v_\ep\|_{L^2} \\
	&\le
	\|g\|_{Lip}  C(c_1,c_2, n) . 
\end{align*}
Thus 
\[
	\frac d{dt} T^g_r(v_\ep(\cdot, \cdot, t)) \le 
	C(r, n,\|g\|_{H^1})+\|g\|_{Lip}  C(c_1,c_2, n)  =: C_4.
\]
It follows that \eqref{eq:bardet1} holds for all $t\in (0,t^*)$. Then, thanks to
\eqref{eq:bardet2} and \eqref{eq:bardet3}, we conclude that $t^*\ge t_4$,
completing the proof of \eqref{eq:bardet1}.
\end{proof}

\subsection{Proof of Corollary \ref{cor:1}}

\begin{proof}
Since $f(0)$ may not be a Lipschitz function, we first mollify it to a function
	which we call $g$ and which we require to satisfy $\sup_{i,z}|f_i(0,z) -
	g_i(z)| < \alpha \rho_{f(0)}$ for some $\alpha<1/8$ to be chosen, and thus
	$\rho_g > (1-2\alpha)\rho_{f(0)}.$ Since $f(0)$ is already in $H^1$, we have
	that $|g|_{H^1} \leq |f(0)|_{H^1}$. Proposition \ref{prop:4}, applied to
	$g$, $r=\rho_g/4$, provides us with constants $\ep_4,t_4,c_4,C_4$, the
	important point being that $\ep_4$ and $c_4$ do \emph{not} depend on the
	strength of the mollification. Without loss of generality, we may also
	assume that $c_4 \le \frac 18 \rho_{f(0)}$. In view of the assumptions of
	Theorem \ref{T.main}, we may assume, decreasing the value of $\ep_4$ if
	necessary, that \eqref{eq:barguil1} and \eqref{eq:barguil2} hold for every
	$\ep\leq \ep_4.$ Finally, it is clear that $|| \chi^g_{r,\ep}(\cdot,
	z)\|_{W^{1,\infty}(\omega)} \le C(r)h_\ep^{-1}$ for every $z\in (0,L)$, so
	assumption \eqref{mainthm.h1} implies that $\limsup_{\ep\to
	0}T^g_{r,\ep}(u_\ep^0) \le \pi \|f(0)-g\|_{L^2}^2$.  We may therefore
	assume, decreasing $\ep_4$ further if necessary, that
	$T^g_{r,\ep}(u^0_\ep)\le 2\pi  \|f(0)-g\|_{L^2}^2 \le 2n\pi^2\alpha^2L
	\rho_{f(0)}^2$ for every $\ep\leq \ep_4$, and in particular that
	\eqref{eq:barguil3} holds.  In view of \eqref{eq:barguil1} and
	\eqref{eq:bardet3}, we may then apply Proposition \ref{P:CJ} for each fixed
	time $t \in [0,t_4]$ and derive some limiting $f^*(t)$ after passing to a
	possible subsequence. 

The potential difficulty at this level is that the subsequence may depend on the
	value of $t$; to overcome this we will rely on the form of continuity in
	time provided by estimate \eqref{eq:bardet1}. We first derive some estimates
	that apply to any limit $f^*(t)$ produced by the above argument. Note that
	\eqref{eq:bardet2} and \eqref{comp} imply that 
\[
	\frac 1 {h_\ep}\int_0^L \|  \pi\sum_{i=1}^n \delta_{h_\ep f^*_i(z,t)}
	- \pi\sum_{i=1}^n  \delta_{h_\ep g_i(z)}\|_{W^{-1,1}(\omega)} dz  
	\le  \left( n\pi L (T^g_{r,\ep}(u_\ep^0)+C_4 t)\right)^\frac12,
\]
and \eqref{fbounds} implies that $\| f^*(t)\|_{H^1}\le C(G_0(g))$. Using
	\eqref{BCL},
\[
	\int_0^L \min_{\sigma\in S_n} | f^*_{\sigma(i)}(z,t) - g_i(z)| dz \
	\le   \left(n \pi L (T^g_{r,\ep}(u_\ep^0)+C_4 t)\right)^\frac12  .
\]
Since $f^*(t)-g$ is uniformly bounded in $H^1$, by choosing  $t_0\le t_4$ and
	$\alpha$ sufficiently small, we conclude that 
\begin{align*}
	\max_z \min_{\sigma\in S_n} | f^*_{\sigma(i)}(z,t) - f^0_i(z)| 
	&<
	\max_z\min_{\sigma\in S_n} | f^*_{\sigma(i)}(z,t) 
	- g_i(z)|  + |f^0_i(z)-g_i(z)|
	\\ &\le \frac 1{16}\rho_{f(0)}
\end{align*}
for all $t\in [0,t_0]$.  It follows that there is a single permutation $\sigma$
	that attains the min for all $z$. After relabelling $f^*$ if necessary, we
	deduce that \eqref{whoisbardet} holds when $s=0$. Finally, using the
	$L^\infty$ continuity of $s\mapsto f(\cdot, s)$ and decreasing $t_0$ as
	needed, we deduce that \eqref{whoisbardet} holds for all $s,t\in [0,t_0]$.

To prove continuity in time, we start by using a Cantor diagonal argument to fix
	a subsequence $\eps \to 0$ such that
$$
	\int_0^L \| J_x u_\ep(\cdot ,z,h_\ep^2t) 
	- \pi\sum_{j=1}^n \delta_{h_\ep f^*_j(z,t)}\|_{W^{-1,1}(\omega)} dz 
	= o(h_\ep) \qquad \mbox{ as }\ep\to 0
$$
for every time $t$ in $\Q \cap [0,t_0]$. We claim that the mapping $t \mapsto
f^*(t)$ is uniformly continuous from $\Q \cap [0,t_0]$ into $L^1([0,L]).$
Indeed, let $\eta>0$ be given, and let $s_0,s_1 \in \Q \cap [0,t_0]$ be
arbitrary.  We write 
\begin{equation}\label{eq:uran1}
	\sum_i \|f_i^*(s_0)-f_i^*(s_1)\|_{L^1} 
	\leq \sum_i \|f_i^*(s_0)-g_i^*(s_0)\|_{L^1} + 
	\sum_i \|g_i^*(s_0)-f_i^*(s_1)\|_{L^1}
\end{equation}
where $g^*(s_0)$ is a mollification of $f^*(s_0)$. It follows from
\eqref{fbounds} that  $t\mapsto f^*(t)$ is uniformly bounded with values into
$H^1$ , so we may fix the mollification parameter sufficiently fine, but
independently of $s_0$, such that 
\begin{equation}\label{eq:uran2}
	\sum_i \|f_i^*(s_0)-g_i^*(s_0)\|_{L^1} \leq \eta/2. 
\end{equation}
Next, we pass to the limit in the conclusions of Proposition \ref{prop:4}
applied this time to $g=g^*(s_0)$ and conclude that 
\begin{equation}\label{eq:uran3}
	\begin{split}
	&\pi\sum_i \|g_i^*(s_0)-f_i^*(s_1)\|_{L^1}\\
	&= \lim_{\ep \to 0}h_\ep^{-1} \int_0^L \|J_x u_\ep(\cdot,z,h_\ep^2s_1) 
		- \pi\sum_{i=1}^n \delta_{h_\ep g^*_i(z,s_0)}\|_{W^{-1,1}(\omega)} dz\\
	&\le \lim_{\ep \to 0}  \left(n \pi L (T^{g^*(s_0)}_{r,\ep}(u_\ep^0)
		+ C_4 |s_1-s_0|)\right)^\frac12\\
	&\le \left( n \pi L ( \pi \|g^*(s_0)-f^*(s_0)\|_{L^2}^2
		+C_4 |s_1-s_0|)\right)^\frac12,
\end{split}\end{equation}
where $C_4$ depends only on the mollification parameter. (We have implicitly
	used  the fact that components of $f^*$ have been labelled correctly, as
	reflected in \eqref{whoisbardet}.) We therefore further decrease the
	mollification parameter if necessary, yet independently of $s_0$, so that
	$n\pi^2 L \|g^*(s_0)-f^*(s_0)\|_{L^2}^2 \leq \eta^2/32.$ Once this, and
	hence $C_4$ are fixed, we require $|s_0-s_1|$ to be small enough so that
	$n\pi LC_4 |s_1-s_0| \leq \eta^2/32.$ Combining \eqref{eq:uran2} and
	\eqref{eq:uran3} in \eqref{eq:uran1} yields the uniform continuity of $f^*.$
	In the sequel we denote still by $f^*$ the unique continuous extension of
	$f^*$ to the whole interval $[0,t_0].$ We claim that the conclusion of
	Corollary \ref{cor:1} holds for any $t \in [0,t_0],$ with no need of further
	subsequences. Indeed, this follows from the fact that for each fixed $t$ in
	$[0,t_0]$ there exist at least some further subsequence for which the
	convergence to some $f^{**}(t)$ holds (this is by Proposition \ref{P:CJ} as
	we already saw it), and on the other hand by our previous argument (equally
	applied to the countable set $(\Q \cap [0,t_0]) \cup \{t\}$) the only
	possible limit along \emph{any} such subsequence is necessarily equal to
	$f^*(t).$
\end{proof}
\subsection{Proof of Proposition \ref{P:jstar}}
\begin{proof}
For $r,R>0$, define 
\[
	\mathcal G_{r,R } :=  \{ (t,x,z)\in [0,t_0]\times B(R) \times [0,L] : 
	|x - f_k^*(z,t)|\ge r, \ k=1,\ldots, n \}.
\]

Given $\mathcal O$ as in the statement of the Proposition, we may fix $r,R>0$
	such that $\mathcal O \subset \mathcal G_{r,R}$.

We will only consider $\ep$ small enough that $B(R)\subset \omega_\ep$.
It is then rather clear that 
\[
	j_{\omega_\ep}^*(f^*(t)) \rightarrow j^*_{\R^2}(f^*(t)) \ \ 
	\mbox{ locally uniformly on } \mathcal G_{r,R}\mbox{ for every }r>0.
\]
It thus follows from Proposition \ref{prop:3}  (with $\Sigma=\delta =1$,
rewritten in terms of $v_\ep$) that 
\[
	\left\| \frac  {jv_\ep}{|v_\ep|}  - j^*_{\R^2} \right\|_{\mathcal G_{r,R}} 
	\le C
\]
for all sufficiently small $\ep$, where $C$ is independent of $r$ and $R$.  By
	extracting weak limits and employing a Cantor diagonal argument, we conclude
	that there exists  a vector field $H\in L^2([0,t_0]\times \R^2\times \TL)$
	such that
\[
	\frac  {jv_\ep}{|v_\ep|}  - j^*_{\R^2}\rightharpoonup H
	\mbox{ weakly in }L^2(\mathcal G_{r,R})\mbox{ for every }r,R>0.
\]

Now fix $\vp\in \calD((0,t_0)\times \R^2\times \TL)$ and compute, for $\ep$
	sufficiently small,
\begin{align}
	\left|\int \nabla_x^\perp \vp \cdot \frac {jv_\ep}{|v_\ep|}  -
	\int \nabla_x^\perp \vp \cdot  jv_\ep \right|
	&\le
	\int |\nabla_x^\perp \vp| \,  \left|\frac {jv_\ep}{|v_\ep|}\right| 
	\, \left|1-|v_\ep|  \right |
	= o(1)
	\label{getridofv}
\end{align}
as $\ep\to 0$, in view of  the pointwise inequality $ \left|\frac
	{jv_\ep}{|v_\ep|}\right| \, \left|1-|v_\ep|  \right | \le \ep e_\ep(v_\ep)$
	and the energy bound on $v_\ep$. Next, integrating by parts and using
	Corollary \ref{cor:1} and the definition of $j^*_{\R^2}$, 
\[
	\int \nabla_x^\perp \vp \cdot  jv_\ep
	= 
	2 \int \vp Jv_\ep \to 
	\int_0^{t_0} \int_0^L \sum_{i=1}^n \vp (f_i(z),z)\, dz\, dt
	=
	\int \nabla_x^\perp \vp \cdot  j^*_{\R^2}(f^*) .
\]
By combining these and using the fact that $H\in L^2$, which implies that the
	singularities along $\{ (t, f_i(z),z) : t\in [0,t_0], z\in [0,L], i=1,\ldots
	n \}$ are removable, we infer that $\nabla^\perp\cdot H = 0$ on  $\R \times
	\R^2\times \TL$. Similarly, by \eqref{massv}, 
\[
	\int \nabla_x \vp \cdot  jv_\ep
	= 
	-\int \partial_t\vp (|v_\ep|^2-1) + h_\ep^2 \partial_z \vp \, j_zv_\ep 
	\to 0,
\] 
since $(v_\ep|^2-1)^2 \le 4 \ep^2e_\ep(v_\ep)$ and
$
	\left|h_\ep^2 \partial_z \vp \, j_zv_\ep \right|
	\le h_\ep |\partial_z\vp| (\frac{ |\partial_z v_\ep|^2}\logeps + |v_\ep|^2 ),
$
together with  \eqref{dzu2}, rescaled to read $\| \nabla
v_\ep(t)\|_{L^2(dx\,dz)}^2 \le C\logeps$ for every $t\in [0,t_0]$. Arguing as in
\eqref{getridofv} to eliminate the factor of $|v_\ep|$ in the denominator and
recalling that $\nabla_x\cdot j^*_{\R^2}(f^*)=0$ by definition, we conclude that
\[
	\int \nabla_x \cdot (\frac  {jv_\ep}{|v_\ep|}  - j^*_{\R^2}) \, \rightarrow 0,
\]
and hence that $\nabla_x \cdot H = 0$ in $\calD'$. We conclude by applying Lemma
\ref{Liouville} below to the vector field  $w(t,x,z) = \zeta(t)H(t,x,z)$, where
$\zeta$ is an arbitrary function with compact support in $[0,t_0]$. 
\end{proof}

The proof of Proposition \ref{P:jstar} used the following 

\begin{lemma}\label{Liouville}
Assume that $w\in L^2(\R \times \R^2\times \TL)$ satisfies
\begin{equation}
	\nabla_x\cdot w = 0, \qquad \nabla_x^\perp \cdot w = 0 
	\qquad\mbox{ in }\mathcal D' .
	\label{Lp1}
\end{equation}
Then $w=0$.
\end{lemma}

\begin{proof}
If $w$ is smooth, then since $\nabla_x^\perp \cdot w=0$, we may write $w =
	\nabla_x f$ for some scalar function $f$. Then the fact that $\nabla\cdot w
	= 0$ implies that $f$ is harmonic, and hence that $w$ is harmonic. For {\em
	a.e.} $t\in \R$ and $z\in \TL$, 
\[
	\int_{\R^2} |w(t,x,z)|^2\, dx = 0,
\]
so Liouville's Theorem implies that $w(t,\cdot,z)=0$ for such $(t,z)$,and
	therefore everywhere in $\R\times \R^2\times \TL$. 

If $w$ is not smooth, then we fix an approximate identity $(\eta_\ep)$, and we
	write $w_\ep := \eta_\ep * w$. Then $w_\ep$ satisfies conditions
	\eqref{Lp1}, with $\|w_\ep\|_{L^2}\le \|w\|_{L^2} <\infty$ for every
	$\ep>0$, and $w_\ep \to w$ in $L^2$, so it follows that $w=0$ {\em a.e.} 
\end{proof}
\smallskip
\noindent

\medskip\noindent
{\it Acknowledgements.} This work was partially supported by the ANR project ODA
(ANR-18-CE40-0020-01) of the Agence Nationale de la Recherche, and
by the Natural Sciences and Engineering Research Council of Canada under Operating Grant 261955.

\end{document}